\documentclass [leqno,11pt]{amsart}
\usepackage{amssymb,amsmath, amsthm, color}

\numberwithin{equation}{section}

\allowdisplaybreaks

\let\b=\beta

\let\r=\rho

\let\f=\frac

\let\om=\omega

\let\Om=\Omega

\let\na=\nabla

\let\pa=\partial
\let\ph=\varphi

\def\div{\mbox{div}}

\def\curl{\mathop{\rm curl}\nolimits}
\def\Ds{\langle \nabla\rangle^{s- \frac{1}{2}}}

\def\bbT{\mathbb{T}}

\newcommand{\newcom}{\newcommand}
\newcommand{\vc}[1]{{\bf #1}}
\newcom{\BT}{{\mathbb{T}^2}}
\newcom{\ve}{\vc{e}}
\newcom{\vN}{\vc{N}}
\newcom{\vn}{\vc{n}}
\newcom{\vG}{\vc{G}}
\newcom{\vF}{\vc{F}}
\newcom{\vf}{\vc{f}}
\newcom{\vg}{\vc{g}}
\newcom{\vq}{\vc{q}}
\newcom{\vu}{\vc{u}}
\newcom{\vv}{\vc{v}}
\newcom{\vV}{\vc{V}}
\newcom{\vw}{\vc{w}}
\newcom{\vW}{\vc{W}}
\newcom{\vb}{\vc{b}}
\newcom{\vh}{\vc{h}}
\newcom{\vz}{\vc{z}}
\newcom{\vup}{\vu^{+}}
\newcom{\vum}{\vu^{-}}
\newcom{\vvp}{\vv^{+}}
\newcom{\vvm}{\vv^{-}}
\newcom{\vbp}{\vb^{+}}
\newcom{\vbm}{\vb^{-}}
\newcom{\vhp}{\vh^{+}}
\newcom{\vhm}{\vh^{-}}
\newcom{\Omp}{{\Om^+}}
\newcom{\Omm}{{\Om^-}}
\newcom{\vupm}{{\vu^{\pm}}}
\newcom{\vvpm}{{\vv^{\pm}}}
\newcom{\vbpm}{{\vb^{\pm}}}
\newcom{\vhpm}{{\vh^{\pm}}}
\newcom{\vFpm}{{\vF^{\pm}}}
\newcom{\vwp}{{\vc{w}^+}}
\newcom{\vwm}{{\vc{w}^-}}
\newcom{\vwpm}{{\vc{w}^{\pm}}}
\newcom{\Ompm}{{\Omega^{\pm}}}
\newcom{\vom}{\boldsymbol{\omega}}
\newcom{\vvap}{\boldsymbol{\varpi}}
\newcom{\vop}{\vom^{+}}
\newcom{\vnu}{\boldsymbol{\nu}}
\newcom{\vopm}{\vom^{\pm}}
\newcom{\vjp}{\vj^+}
\newcom{\vjm}{\vj^-}
\newcom{\vjpm}{\vj^{\pm}}
\newcom{\vj}{\boldsymbol{\xi}}
\newcommand{\beq}{\begin{equation}}
\newcommand{\eeq}{\end{equation}}
\newcommand{\ben}{\begin{eqnarray}}
\newcommand{\een}{\end{eqnarray}}
\newcommand{\beno}{\begin{eqnarray*}}
\newcommand{\eeno}{\end{eqnarray*}}
\def\eqdefa{\buildrel\hbox{\footnotesize def}\over =}

\newtheorem{theorem}{Theorem}[section]
\newtheorem{definition}[theorem]{Definition}
\newtheorem{lemma}[theorem]{Lemma}
\newtheorem{proposition}[theorem]{Proposition}

\newtheorem{remark}[theorem]{Remark}

\begin{document}

\title[Free boundary problem in incompressible elastodynamics]{Well-posedness of the free boundary problem in incompressible elastodynamics}

\author{Hui Li}
\address{School of  Mathematical Sciences, Peking University, Beijing 100871, China}
\email{lihui92@pku.edu.cn}

\author{Wei Wang}
\address{Department of Mathematics, Zhejiang University, Hangzhou 310027, China}
\email{wangw07@zju.edu.cn}

\author{Zhifei Zhang}
\address{School of  Mathematical Sciences, Peking University, Beijing 100871, China}
\email{zfzhang@math.pku.edu.cn}

\date{\today}

\begin{abstract}
In this paper, we prove the local well-posedness of the free boundary problem in incompressible elastodynamics under a natural stability condition,
which ensures that the evolution equation describing the free boundary is strictly hyperbolic.
Our result gives a rigorous confirmation that the elasticity has a stabilizing effect on the Rayleigh-Taylor instability.

\end{abstract}

\maketitle


\section{introduction}
\subsection{Presentation of the problem}
In this paper, we consider the incompressible inviscid flow in 3-D elastodynamics:
\begin{equation}\label{elso}
	\left\{
	\begin{array}{l}
\pa_t \rho +  \vu\cdot\nabla\rho=0,\\
		\pa_t (\rho\vu) +  \vu\cdot\nabla(\rho\vu)+ \nabla p= \div(\rho\vF\vF^\top), \\
		\pa_t  \vF +  \vu\cdot\nabla \vF = \nabla\vu \vF,\\
		\div \vu = 0, \\
	\end{array}
	\right.
\end{equation}
where $\rho$ is the density of fluids, $\vu(t,x)=(u_1,u_2,u_3)$ denotes the fluid velocity,
$p(t,x)$ is the pressure, $\vF(t,x)=(F_{ij})_{3\times 3}$ is the deformation tensor, $\vF^\top=(F_{ji})_{3\times 3}$ denotes the transpose of the matrix $\vF$,
$\vF\vF^\top$ is the Cauchy-Green tensor in the case of neo-Hookean elastic materials,
$(\nabla\vu)_{ij}=\pa_ju_i$, $(\nabla\vu \vF)_{ij}=\sum^3_{k=1}\vF_{kj}\partial_k u_i$,
$(\div \vF^\top)_i=\sum^3_{j=1}\pa_jF_{ji}$,
$(\div\vF\vF^\top)_i=\sum^3_{j,k=1}\pa_j(F_{ik}F_{jk})$.

We study the solution of (\ref{elso}) which are smooth on each side of a smooth interface $\Gamma(t)$
in a domain $\Omega$. Precisely, for simplicity, we let
\begin{align*}
&	\Om=\mathbb{T}^2\times[-1,1]\subset \mathbb{R}^3,\quad  \Gamma(t)=\{x\in\Om|x_3=f(t,x'),x'=(x_1,x_2)\in\bbT^2\},\\
& \Om_t^\pm=\{x\in\Om|x_3\gtrless f(t,x'),x'=(x_1,x_2)\in\bbT^2\},\qquad Q_T^\pm=\underset{t\in(0,T)}{\bigcup}\{t\}\times\Om^\pm_t.
\end{align*}
We consider that $\rho_{\Om_t^\pm}=\rho^\pm$ are two constants, and
\begin{align*}
	\vu^\pm:=\vu|_{\Om^\pm_t},\qquad \vF^\pm:=\vF|_{\Om^\pm_t},\qquad p^\pm:=p|_{\Om^\pm_t},
\end{align*}
are smooth in $Q_T^\pm$ and satisfy
\begin{equation}\label{els}
	\left\{
	\begin{array}{ll}
		\rho^\pm\big(\pa_t \vupm + \vupm\cdot\nabla \vupm\big)+ \nabla p^{\pm} =  \rho^\pm\sum\limits^3_{j=1}(\vF_j^\pm\cdot\nabla) \vF_j^\pm &\text{ in }\quad Q^{\pm}_T,  \\
		\div \vupm = 0,\,\div \vF^{\pm\top} = 0&\text{ in }\quad  Q^{\pm}_T,  \\
		\pa_t  \vF_j^\pm +  \vu^\pm\cdot\nabla \vF_j^\pm =(\vF_j^\pm\cdot\nabla)\vu^\pm&\text{ in }\quad Q^{\pm}_T,
	\end{array}
	\right.
\end{equation}
with the boundary conditions on the moving interface $\Gamma_t$:
\begin{align}\label{elsp}
	[p]\eqdefa p^+-p^-=0,\quad \vupm\cdot\vn =V(t,x), \quad \vF_j^\pm\cdot\vn = 0.
\end{align}
Here $\vF^\pm_j=(F^\pm_{1j},F^\pm_{2j},F^\pm_{3j})$, $\vn$ is the outward unit normal to $\pa\Om^-_t$ and $V(t,x)$ is the normal velocity of $\Gamma_t$.
On the artificial boundary $\Gamma^\pm=\bbT^2\times\{\pm1\}$, we impose the following boundary conditions on ($\vu^\pm$, $\vF^\pm$):
\begin{align}\label{vorsheet:top-bot-bc}
	u_3^\pm=0,\qquad F_{3j}^\pm=0\qquad \text{on}~\Gamma^\pm.
\end{align}

The system (\ref{els}) is supplemented with the initial data
\begin{align}\label{elsi}
	\vupm(0,x)=u _0^\pm(x),\qquad \vF^\pm(0,x)=\vF_0^\pm\qquad \text{in}~\Om_0^\pm,
\end{align}
where the initial data satisfies
\begin{equation}\label{elsii}
  \left\{
  	\begin{array}{ll}
  		\div\vu^\pm_0=0,~\div\vF^\pm_{j,0}=0&\text{in}~\Om_0^\pm,\\
  		\vu^+_0\cdot\vn_0=\vu^-_0\cdot\vn_0,~\vF^\pm_{j,0}\cdot\vn_0=0&\text{on}~\Gamma_0.
  	\end{array}
  \right.
\end{equation}
The system (\ref{els})-(\ref{elsi}) is called the vortex sheet problem for incompressible elastodynamics.
One of main goals in this paper is to study the local well-posedness of this system under some suitable stability conditions imposed on the initial data.

In our setting, the boundary condition on $\Gamma_t$ in (\ref{els}) is transformed into
\begin{align*}
	[p]=0,\qquad \vu^\pm\cdot \vN=\pa_tf,\qquad \vF_j^\pm\cdot \vN=0\qquad \text{on}~\Gamma_t,
\end{align*}	
where $\vN=(-\pa_1f,-\pa_2f,1)$ and $\vn= {\vN}/{|\vN|}.$

Let us remark that the divergence free restriction on $\vF_j^\pm$ is automatically satisfied if $\div\vF^\pm_{j,0}=0$. Indeed, if we apply the divergence operator to the third equation of (\ref{els}), we will deduce the following transport equation
\begin{align*}
	\pa_t\div\vF^\pm_j + \vupm\cdot\nabla\div\vF^\pm_j=0.
\end{align*}
Similar argument can be also applied to yield that $\vF^\pm_j\cdot\vN = 0$ if $\vF^\pm_{j,0}\cdot\vN_0 = 0$.

For a special case $\rho^+=0$, the problem reduces to another type of free boundary problem for idea incompressible elastodynamics,
that is,
\begin{equation}\label{elsf}
	\left\{
	\begin{array}{ll}
		\pa_t \vu+  \vu\cdot\nabla \vu+ \nabla p= \sum\limits^3_{j=1}\vF_j\cdot\nabla \vF_j&\text{in}\qquad \Om_t^-, \\
		\div \vu = 0,\quad\div \vF^\top=0&\text{in}\qquad \Om_t^-,\\
		\pa_t  \vF_j +  \vu\cdot\nabla \vF_j = \vF_j\cdot\nabla\vu&\text{in}\qquad \Om_t^-,
	\end{array}
	\right.
\end{equation}
where $\vF_j=(F_{1j},F_{2j},F_{3j})$ and $\Om_t^-$ and $\Gamma_t$ are defined as above.
On the free boundary $\Gamma_t$, the boundary conditions are given by
\begin{align}\label{freebp:interface-bc}
	p=0,\qquad \vu\cdot \vN=\pa_tf, \qquad \vN\cdot\vF_j=0\qquad \text{on}\quad \Gamma_t,
\end{align}
while on the bottom boundary $\Gamma^-$, it holds that
\begin{align}\label{freebp:bot-bc}
	u_3=0,\qquad F_{3j}=0\qquad \text{on}~\Gamma^-.
\end{align}	
This system is supplemented with initial data
\begin{align}\label{elsfi}
	\vu(0,x)=u_0(x),\qquad \vF(0,x)=\vF_0\qquad \text{in}~\Om_0^-.
\end{align}

\subsection{Background}

For incompressible inviscid flow, the Kelvin-Helmholtz instability has been known for over a century \cite{Maj}. It was well known that the surface tension can stabilize the Kelvin-Helmholtz and Rayleigh-Taylor instability \cite{AM, CCS, SZ2}.  Syrovatskij \cite{Sy} and Axford \cite{Ax} found that the magnetic field has a stabilization effect on the Kelvin-Helmholtz instability. Recently, there are many important works devoted to confirming this stabilizing mechanism. For the current-vortex sheet problem, we refer to \cite{Tra1, Tra2, Chen, WY} for compressible case and \cite{MTT1, Tra-in, CMST, SWZ1} for incompressible case. For plasma-vacuum problem, we refer to \cite{Tra-JDE, ST} for compressible case and \cite{MTT2, SWZ2} for incompressible case. We also refer to some related works \cite{HL, Hao, GW} on the incompressible  plasma-vacuum problem.

For the inviscid  elastodynamics, there are several recent progress on the free boundary problems. For 2-D compressible vortex sheet problem in elastodynamics, Chen-Hu-Wang \cite{CHW} analyzed the linearized stability and proved the stabilization effect of elasticity on vortex sheets. In \cite{Tra3}, Trakhinin proved the well-posedness of the one-fluid free boundary problem in compressible elastodynamics under the condition that there are two columns of the $3\times 3$ deformation tensor which are non-collinear at each point of the initial surface. For the incompressible case, Hao-Wang \cite{HW} proved a priori estimates for solutions in Sobolev spaces under the Rayleigh-Taylor sign condition.

The aim of this paper is to show the local well-posedness for both two free boundary problems in incompressible elastodynamics under a natural stability condition by using the  method developed in \cite{SWZ1}. The basic idea is to derive an evolution equation describing the free boundary
so that it is strictly hyperbolic under a suitable stability condition. This idea is very effective to study the free boundary problems of the incompressible Euler equations \cite{Wu1, Wu2, ZZ, SZ1}.

\subsection{Main results}
To ensure the stability of the system (\ref{els})-(\ref{elsi}) and the system (\ref{elsf})-(\ref{elsfi}), certain stability conditions are required. In this paper, we assume the following stability condition for (\ref{els})-(\ref{elsi}):
\begin{align}\label{condition:s1}
\inf_{\Gamma_t} \,\,\,\inf_{\xi\in\mathbb{S}^2,\,\xi\cdot\vN=0} \Big\{(\rho^++\rho^-)\Big[\rho^+(\xi\cdot{\vF}^+)^2+\rho^-(\xi\cdot{\vF}^-)^2\Big]-\rho^+\rho^-(\xi\cdot[\vu])^2\Big\}>0,
\end{align}
where $\xi\cdot\vF=(\xi_iF_{ij})_{1\le j\le 3}$ and   $[\vu]=\frac{1}{2}(\vu^+-\vu^-)$ on $\Gamma_f$. (\ref{condition:s1}) is equivalent to
\begin{align}\label{condition:s2}
	\Lambda(\vFpm,\vv)\eqdefa&\inf_{x\in\Gamma_t}\inf_{\ph_1^2+\ph_2^2=1}\sum^3_{j=1}\Big(\frac{\rho^+}{\r^++\r^-}(F^+_{1j}\ph_1+F^+_{2j}\ph_2)^2
+\frac{\rho^-}{\r^++\r^-}(F^-_{1j}\ph_1+F^-_{2j}\ph_2)^2\Big)\\ \nonumber
		&-(v_1\ph_1+v_2\ph_2)^2>0,
\end{align}
where $(v_1,v_2,v_3)=\frac{\sqrt{\r^+\r^-}}{\r^++\r^-}[\vu]$. Our first main result is stated as follows.
\begin{theorem}\label{thm:1}
 Let $s\ge3$ be an integer and assume that
	\begin{align*}
		f_0\in H^{s+ \frac{1}{2}}(\bbT^2),\quad \vu^\pm_0,\,\vF^\pm_0\in H^s(\Om_0^\pm).
	\end{align*}
Furthermore we assume that there exists $c_0>0$ so that
	\begin{itemize}
		\item[1.] $-(1-2c_0)\le f_0\le (1-2c_0)$;
		\item[2.] $\Lambda(\vF_0^\pm,\vv_0)\ge 2c_0$.
	\end{itemize}
	Then there exists $T>0$ such that the system (\ref{els}) admits a unique solution $(f, \vu, \vF)$ in $[0,T]$ satisfying
	\begin{itemize}
		\item[1.] $f\in L^\infty([0,T), H^{s+\frac{1}{2}}(\bbT^2))$;
		\item[2.] $\vu^\pm,\,\vF^\pm\in L^\infty\big(0,T;H^{s}(\Omega^\pm_t)\big)$;
		\item[3.] $-(1-c_0)\le f\le (1-c_0)$;
		\item[4.] $\Lambda(\vF^\pm,\vv)\ge c_0$.
	\end{itemize}
\end{theorem}
When $\rho^+=0$, the stability condition (\ref{condition:s1}) reduces to $(\xi\cdot{\vF}^-)^2>0$ on $\Gamma_t$
for any $\xi\in\mathbb{S}^2$ with $\xi\cdot\vN=0$, which is equivalent to $\text{rank}(\vF)=2.$
Therefore, as a corollary of Theorem \ref{thm:1}, we have the following result which concerns the well-posedness for the system (\ref{elsf})-(\ref{elsfi}).
\begin{theorem}\label{thm:2}
 Let $s\ge3$ be an integer and assume that
	\begin{align*}
		f_0\in H^{s+ \frac{1}{2}}(\bbT^2),\quad \vu_0,\,\vF_0\in H^s(\Om_0^\pm).
	\end{align*}
		Furthermore we assume that there exists $c_0>0$ so that
	\begin{itemize}
		\item[1.] $-(1-2c_0)\le f_0\le (1-2c_0)$;
		\item[2.] $\rm{rank}(\vF_0)=2$ on $\Gamma_0$.
	\end{itemize}
	Then there exists $T>0$ such that the system (\ref{elsf}) admits a unique solution $(f, \vu, \vF)$ in $[0,T]$ satisfying
	\begin{itemize}
		\item[1.] $f\in L^\infty([0,T), H^{s+\frac{1}{2}}(\bbT^2))$;
		\item[2.] $\vu,\,\vF\in L^\infty\big(0,T;H^{s}(\Omega^\pm_t)\big)$;
		\item[3.] $-(1-c_0)\le f\le (1-c_0)$;
		\item[4.] $\rm{rank}(\vF)=2$ on $\Gamma_t$.
	\end{itemize}
\end{theorem}
\begin{remark}
We remark that the assumption $\rm{rank}(\vF_0)=2$  on $\Gamma_0$ is weaker than the assumption proposed
by Trakhinin \cite{Tra3}, which says that among the three vectors $\vF_1$, $\vF_2$ and $\vF_3$ there are two which are non-collinear at each point
of $\Gamma_0$.
There is another type of stability condition which would also ensure the existence of solutions to the system (\ref{elsf})-(\ref{elsfi}) :
\begin{align}
-\vN\cdot\nabla p>0, \quad \text{ on }\Gamma_t,
\end{align}
see \cite{HW} for a priori estimates results.
One would also be interested in studying the wellposedness under the following mixed type of stability condition:
\begin{align}
\{ x\in \Gamma_t: {\rm{rank}}(\vF(x))=2 \}\cup \{ x\in \Gamma_t: -\vN\cdot\nabla p>0 \} =\Gamma_t.
\end{align}
These cases can also be handled in this framework, which will be left in a forthcoming work.
\end{remark}

\begin{remark}
Our method could be applied to 2-D case, which in particular means that the elasticity has a stabilization effect on the Rayleigh-Taylor instability.
Indeed, for the 2-D case, we have that $\vF^\pm$ are $2\times 2$ matrices, and $\vF_j^\pm\cdot\vN=0, [\vu]\cdot\vN=0$, which implies that
$\vF_j^\pm, [\vu]$ are collinear to each other. Therefore, the stability condition (\ref{condition:s1}) for  (\ref{els})-(\ref{elsi}) reduces to
\begin{align}
(\rho^++\rho^-)\Big[\rho^+|{\vF}^+|^2+\rho^-|{\vF}^-|^2\Big]-\rho^+\rho^-|[\vu]|^2>0,
\end{align}
and for the system (\ref{elsf})-(\ref{elsfi}), the stability condition $\text{rank}(\vF)=2$ reduces to
\begin{align}
  |\vF|>0.
\end{align}
The solutions can be constructed in a similar way as Theorem \ref{thm:1} and \ref{thm:2}.

\end{remark}

The rest of this paper is organized as follows. In Section 2, we will introduce the reference domain, harmonic coordinate, and the Dirichlet-Neumann operator. In Section 3, we reformulate the system into a new formulation. In Section 4, we present the uniform estimates for the linearized system. In Section 5, we prove the existence and uniqueness of the solution. In Section 6, we present a sketch of the proof of Theorem \ref{thm:2}.

\section{Reference domain, harmonic coordinate and Dirichlet-Neumann Operator}
For free boundary problems, as the domain of the fluid is changing with time $t$, we always draw the moving domain back to a fixed domain which is called reference domain \cite{SWZ1}.

 Let $\Gamma_*$ be a fixed graph given by
\begin{align*}
	\Gamma_*=\{(y_1,y_2,y_3):y_3=f_*(y_1,y_2)\}.
\end{align*}
The reference domain $\Om^\pm_*$ is given by
\begin{align*}
	\Om_*=\bbT^2\times(-1,1),\quad \Om^\pm_*=\{y\in\Om_*|y_3\lessgtr f_*(y_1,y_2)\}.
\end{align*}
We will look for the free boundary which lies in a neighborhood of the reference domain. As a result, we define
\begin{align*}
	\Upsilon(\delta,k)&\eqdefa\Big\{f\in H^k(\mathbb{T}^2): \|f-f_*\|_{H^k(\mathbb{T}^2)}\le \delta \Big\}.
\end{align*}
For $f\in \Upsilon(\delta,k)$, we can define the graph $\Gamma_f$ by
\begin{align*}
	\Gamma_f\eqdefa\left\{x\in \Om_t| x_3=f(t,x'), \int_{\mathbb{T}^2}f(t,x')dx'=0 \right\}.
\end{align*}
The graph $\Gamma_f$ separates $\Omega_t$ into two parts:
\begin{align*}
	\Om_f^{+}=\Big\{ x \in \Om_t| x_3 > f(t,x')\Big\}, \quad \Om_f^{-}=\Big\{ x \in \Om_t| x_3 < f(t,x')\Big\}.
\end{align*}
Let $\vN_f=(N_1,N_2,N_3)$ be the outward normal vector of $\Om_f^-$ where
\begin{align*}
	\vN_f\triangleq(-\partial_1f, -\partial_2f, 1),\quad \vn_f\triangleq\vN_f/\sqrt{1+|\nabla f|^2}.
\end{align*}

Then we need to find the draw back maps. For this purpose, we introduce the harmonic coordinate. Given $f\in \Upsilon(\delta,k)$, we define a map $\Phi_f^\pm:\Omega_*^\pm\to\Omega_f^\pm$
by harmonic extension:
\begin{equation}
	\left\{
	\begin{array}{ll}
		\Delta_y \Phi_f^\pm=0, &y\in \Omega_*^\pm,\\
		\Phi_f^\pm(y',f_*(y'))=(y',f(y')),  &y'\in\mathbb{T}^2,\\
		\Phi_f^\pm(y',\pm1)=(y',\pm1),  &y'\in\mathbb{T}^2.\\
	\end{array}
	\right.
\end{equation}

Given $\Gamma_*$, there exists $\delta_0=\delta_0(\|f_*\|_{W^{1,\infty}})>0$ so that $\Phi_f^\pm$ is a bijection when $\delta\le \delta_0$.
Then we can define an inverse map $\Phi_f^{\pm-1}:\Omega_f^\pm\to\Omega_*^\pm$ such that
\begin{equation}\nonumber
	\Phi_f^{\pm-1}\circ\Phi_f^\pm=\Phi_f^\pm\circ\Phi_f^{\pm-1}=\mathrm{Id}.
\end{equation}

The following properties come from \cite{SWZ1}.

\begin{lemma}\label{lem:basic}
Let $f\in \Upsilon(\delta_0,s-\f12)$ for $s\ge 3$. Then there exists a constant $C$ depending only on $\delta_0$ and  $\|f_*\|_{H^{s-\f12}}$ so that

\begin{itemize}
	\item[1.] If $u\in H^{\sigma}(\Om_f^\pm)$ for $\sigma\in [0,s]$, then
	\begin{align*}
		\|u\circ\Phi_f^\pm\|_{H^\sigma(\Om^\pm_*)}\le C\|u\|_{H^\sigma(\Om_f^\pm)}.		
	\end{align*}
	\item[2.] If $u\in H^{\sigma}(\Om_*^\pm)$ for $\sigma\in [0,s]$, then
	\begin{align*}
		\|u\circ\Phi_f^{\pm-1}\|_{H^{\sigma}(\Om_f^\pm)}\le C\|u\|_{H^\sigma(\Om_*^\pm)}.		
	\end{align*}
	\item[3.] If $u, v\in H^{\sigma}(\Om_*^\pm)$ for $\sigma\in [2,s]$, then
	\begin{align*}
		\|uv\|_{H^\sigma(\Omega_f^\pm)}\le C\|u\|_{H^\sigma(\Omega_f^\pm)}\|v\|_{H^\sigma(\Omega_f^\pm)}.
	\end{align*}
\end{itemize}
\end{lemma}
We will use the Dirichlet-Neumann operator, which maps the Dirichlet boundary value of a harmonic function to its Neumann boundary value. That is to say,
for any $g(x')=g(x_1,x_2)\in H^k(\mathbb{T}^2)$, we denote by $\mathcal{H}_f^\pm g$  the harmonic extension to $\Omega^\pm_f$:
\begin{equation}
	\left\{
	\begin{array}{ll}
		\Delta \mathcal{H}_f^\pm g =0,& x\in \Omega_f^\pm,\\
		(\mathcal{H}_f^\pm g)(x',f(x'))=g(x'),&  x'\in\mathbb{T}^2,\\
		\partial_3\mathcal{H}_f^\pm g(x',\pm1)=0,&  x'\in\mathbb{T}^2.
	\end{array}
	\right.
\end{equation}
Then the Dirichlet-Neumann operator is defined by
\begin{align}
  \mathcal{N}^\pm_fg\overset{def}{=}\mp\vN_f\cdot(\nabla\mathcal{H}^\pm_fg)\big|_{\Gamma_f}.
\end{align}
We will use the following properties from \cite{ABZ, SWZ1}.
\begin{lemma}\label{lem:DN}
	It holds that
	\begin{itemize}
		\item[1.] $\mathcal{N}^\pm_f$ is a self-adjoint operator:
		\begin{align*}
			(\mathcal{N}^\pm_f\psi,\phi)=(\psi,\mathcal{N}^\pm_f\phi),\quad\forall \phi,\, \psi\in H^\frac{1}{2}(\bbT^2);
		\end{align*}
		\item[2.] $\mathcal{N}^\pm_f$ is a positive operator:
		\begin{align*}
			(\mathcal{N}^\pm_f\phi,\phi)=\|\na\mathcal{H}_f^\pm\phi\|_{L^2(\Omega_f)}^2\ge 0,\quad \forall \phi\in H^\frac{1}{2}(\bbT^2);	
		\end{align*}
		Especially, if $\int_{\bbT^2}\phi(x')dx'=0$, there exists $c>0$ depending on $c_0, \|f\|_{W^{1,\infty}}$ such that
		\begin{align*}
			(\mathcal{N}^\pm_f\phi,\phi)\ge c\|\mathcal{H}_f^\pm\phi\|_{H^1(\Omega_f)}^2\ge c\|\phi\|_{H^\frac{1}{2}}^2.
		\end{align*}
		\item[3.] $\mathcal{N}^\pm_f$ is a bijection from $H^{k+1}_0(\bbT^2)$ to $H^{k}_0(\bbT^2)$ for $k\ge 0$, where
		\begin{align*}
			H^{k}_0(\bbT^2)\eqdefa H^k(\bbT^2)\cap\{\phi\in L^2(\bbT^2):\int_{\bbT^2}\phi(x')dx'=0\}.
		\end{align*}
	\end{itemize}
\end{lemma}
We will use $x=(x_1,x_2,x_3)$ or $y=(y_1,y_2,y_3)$ to denote the coordinates in the fluid region, and use $x'=(x_1,x_2)$  or $y'=(y_1,y_2)$ to denote the natural coordinates on the interface or on the top/bottom boundary.
In addition, we will use the Einstein summation notation where a summation from 1 to 2 is implied over repeated index,
while a summation from 1 to 3 over repeated index will be explicitly figured out by the symbol $\sum$  (i.e. $a_ib_i=a_1b_1+a_2b_2, \sum_{i=1}^3a_ib_i=a_1b_1+a_2b_2+a_3b_3$).

For a function $g:\Om\to\mathbb R$, we denote $\nabla g=(\pa_1g,\pa_2g,\pa_3g)$, and for a function $\eta:\bbT^2\to \mathbb R$, $\nabla\eta=(\pa_1\eta,\pa_2\eta)$. For a function $g:\Om^\pm_f\to\mathbb R$, we can define its trace on  $\Gamma_f$, which is denoted by $\underline g(x')$. Thus, for $i=1,2$,
\begin{align*}
	\pa_i\underline g(x')=\pa_i g(x',f(x'))+\pa_3g(x',f(x'))\pa_if(x').
\end{align*}
We denote by $||\cdot||_{H^s(\Om)}$ the Sobolev norm in $\Om$, and by $||\cdot||_{H^s}$ the Sobolev norm in $\bbT^2$.

\section{Reformulation of the problem}
In this section, we derive a new system which is equivalent to the original system (\ref{els})-(\ref{elsi}). The system consists of the evolution equations of the following quantities:
\begin{itemize}
	\item The height function of the interface: $f$;
	\item The scaled normal velocity on the interface: $\theta=\vu^\pm\cdot\vN_f$;
	\item The curl part of velocity and deformation tensor in the fluid region: $\vom^\pm=\nabla\times\vu^\pm$, $\vG_j^\pm=\nabla\times\vF_j^\pm$;
	\item The average of tangential part of velocity and deformation tensor field on top and bottom fixed boundary:
	\begin{align*}
		\beta_i^\pm(t)=\int_{\bbT^2}u_i^\pm(t,x',\pm1)dx',\quad \gamma^\pm_{ij}(t)=\int_{\bbT^2}F_{ij}^\pm(t,x',\pm1)dx'\,\,(i=1,2; j=1,2,3).
	\end{align*}
\end{itemize}

\subsection{Evolution of the scaled normal velocity}
Let
\begin{align}
	\theta(t,x'){=}\vu^\pm(t,x',f(t,x'))\cdot\vN_f(t,x').
\end{align}
Then we have
\begin{align}\label{eq:form:f}
	\partial_tf(t,x')=\theta(t,x').
\end{align}
First of all, one can easily obtain the following elementary lemma, which is useful in the derivation of the evolution of $\theta$.
\begin{lemma}\cite{SWZ1}\label{rel-uh}
	For $\vu=\vupm,\vF^\pm_j$, we have
	\begin{align}\nonumber
		&({\vu}\cdot{\nabla\vu})\cdot\vN_f-{\partial_3u}_jN_j({\vu}\cdot\vN_f)\big|_{x_3=f(t,x')}\nonumber\\
		&=\underline{u}_1\partial_1(\underline{u}_jN_j)+\underline{u}_2\partial_2(\underline{u}_jN_j)
		+\sum_{i,j=1,2}\underline{u}_i\underline{u}_j\partial_i\partial_jf.
		\end{align}
\end{lemma}

Combining the first equation of (\ref{els}) and Lemma \ref{rel-uh}(recall $\vF^\pm_j\cdot\vN_f=0$ on $\Gamma_t$), one can obtain
\begin{align}\nonumber
	\partial_t\theta=&(\partial_t\vu^++\partial_3\vu^+\partial_tf)\cdot\vN_f+\vu^+\cdot\partial_t\vN_f\big|_{x_3=f(t,x')}\\
=&(-\vup\cdot\nabla\vup+\sum^3_{j=1}(\vF_j\cdot\nabla)\vF_j-\nabla p^++\partial_3\vu^+\partial_tf)\cdot\vN_f\nonumber\\
&-
\vu^+\cdot(\partial_1\partial_tf,\partial_1\partial_tf,0)\big|_{x_3=f(t,x')}\nonumber\\\nonumber
=&\big(-(\vup\cdot\nabla)\vup+\partial_3\vu^+(\vup\cdot\vN_f)\big)\cdot\vN_f+\sum^3_{j=1}(\vF_j\cdot\nabla)\vF_j\cdot\vN_f\\
&-\vN_f\cdot\nabla p^+-\vu^+\cdot(\partial_1\theta, \partial_2\theta,0)\big|_{x_3=f(t,x')}\nonumber\\
=&-2(\underline{u}^+_1\partial_1\theta+\underline{u}^+_2\partial_2\theta)-\frac{1}{\r^+}\vN\cdot\underline{\nabla p}^+-\sum^2_{s,r=1}
	\underline{u}_s^+\underline{u}^+_r\partial_s\partial_rf
+\sum^3_{j=1}\sum^2_{s,r=1}\underline{F}^+_{sj}\underline{F}^+_{rj}\partial_s\partial_rf,\label{eq:theta-d}
\end{align}
and similarly,
\begin{align}\label{eq:theta-d-2}
	\partial_t\theta	=&-2(\underline{u}^-_1\partial_1\theta+\underline{u}^-_2\partial_2\theta)-\frac{1}{\r^-}\vN\cdot\underline{\nabla p}^--\sum^2_{s,r=1}
	\underline{u}_s^-\underline{u}^-_r\partial_s\partial_rf
+\sum^3_{j=1}\sum^2_{s,r=1}\underline{F}^-_{sj}\underline{F}^-_{rj}\partial_s\partial_rf.\nonumber
\end{align}

Taking the divergence to the first equation of (\ref{els}), we get
\begin{align}
\Delta p^\pm=\rho^\pm\big(\sum\limits^3_{j=1}\mathrm{tr}(\nabla\vF^\pm_j)^2-\mathrm{tr}(\nabla\vu^\pm)^2\big).
\end{align}
Recall that $\underline{p}^\pm=p^\pm|_{\Gamma_f}$ and $\mathcal{H}^\pm_f$
is the harmonic extension from $\Gamma_f$ to $\Omega^\pm_f$.
Then for the pressure $p^\pm$, we have the following important representation:
\begin{align}
	p^\pm=\mathcal{H}_f^\pm\underline{p}^\pm+\rho^\pm p_{\vupm, \vupm}-\rho^\pm\sum^3_{j=1}p_{\vF^\pm_j, \vF^\pm_j},
\end{align}
where $p_{\vu_1, \vu_2}$ is the solution of elliptic equation
\begin{equation}\label{eqp}
	\left\{
	\begin{array}{ll}
		\Delta p_{\vu_1^\pm, \vu_2^\pm}= -\mathrm{tr}(\nabla\vu_1^\pm\nabla\vu_2^\pm)
		&\text{in}\quad\Omega^\pm_f,\\
		p_{\vu_1^\pm, \vu_2^\pm}=0&\text{on}\quad\Gamma_f,\\
		\ve_3\cdot\nabla p_{\vu_1^\pm, \vu_2^\pm}=0&\text{on}\quad\Gamma^\pm.
	\end{array}\right.
\end{equation}

Thus, from (\ref{eq:theta-d}) and (\ref{eq:theta-d-2}),  we have on $\Gamma_f$ that,
\begin{align*}
	&\frac{1}{\r^+}\vN_f\cdot\nabla \mathcal{H}^+_f\underline{p}^+
	-\frac{1}{\r^-}\vN_f\cdot\nabla \mathcal{H}^-_f\underline{p}^-\\
	&=-\Big[2(\underline{u}^+_1\partial_1\theta+\underline{u} ^+_2\partial_2\theta)+\vN_f\cdot\underline{\nabla(p_{\vup, \vup}-\sum^3_{j=1}p_{\vF^+_j, \vF^+_j})}
	+\sum^2_{s,r=1}(\underline{u} _s^+\underline{u} ^+_r-\sum^3_{j=1}\underline{F} ^+_{sj}\underline{F} ^+_{rj})\partial_s\partial_rf\Big]\\
	&\quad+\Big[2(\underline{u} ^-_1\partial_1\theta+\underline{u} ^-_2\partial_2\theta)+\vN_f\cdot\underline{\nabla(p_{\vum, \vum}-\sum^3_{j=1}p_{\vF^-_j, \vF^-_j})}
	+\sum^2_{s,r=1}(\underline{u} _s^-\underline{u} ^-_r-\sum^3_{j=1}\underline{F} ^-_{sj}\underline{F} ^-_{rj})\partial_s\partial_rf\Big]\\
	&\triangleq -g^++g^-.
\end{align*}
From the definition of DN operator, one has
\begin{align}
	-\frac{1}{\r^+}\mathcal{N}^+_f\underline{p}^+-\frac{1}{\r^-}\mathcal{N}^-_f\underline{p}^-=-g^++g^-.
\end{align}
As  $	\underline{p}^+-\underline{p}^-=0$  on $\Gamma_f$,
we have
\begin{align*}
	\underline{p}^\pm=\widetilde{\mathcal{N}}^{-1}_f(g^+-g^-),
\end{align*}
where
\begin{align*}
	\widetilde{\mathcal{N}}_f=\frac{1}{\r^+}{\mathcal{N}}_f^++\frac{1}{\r^-}{\mathcal{N}}_f^-.
\end{align*}
In addition, we can write
\begin{align*}
	{\mathcal{N}}_f^+=&(\frac{1}{\r^+}+\frac{1}{\r^-})^{-1}(\widetilde{\mathcal{N}}_f+\frac{1}{\r^-}({\mathcal{N}}_f^+-{\mathcal{N}}_f^-)),\\
	{\mathcal{N}}_f^-=&(\frac{1}{\r^+}+\frac{1}{\r^-})^{-1}(\widetilde{\mathcal{N}}_f-\frac{1}{\r^-}({\mathcal{N}}_f^+-{\mathcal{N}}_f^-)),
\end{align*}
which implies
\begin{align*}
	\frac{1}{\r^+}\mathcal{N}^+_f\widetilde{\mathcal{N}}^{-1}_f g^-+\frac{1}{\r^-}\mathcal{N}^-_f\widetilde{\mathcal{N}}^{-1}_fg^+
	=\frac{\r^+g^++\r^-g^-}{\r^++\r^-}-\frac{1}{\r^++\r^-}(\mathcal{N}^+_f-\mathcal{N}^-_f)\widetilde{\mathcal{N}}^{-1}_f(g^+-g^-).
\end{align*}
Consequently, we can obtain
\begin{align}
	\partial_t\theta\nonumber
	=&\,\frac{1}{\r^+}\mathcal{N}^+_f\underline{p}^+-g^+=\frac{1}{\r^+}\mathcal{N}^+_f\widetilde{\mathcal{N}}^{-1}_f(g^+-g^-)-g^+\\\nonumber
	=&-\frac{1}{\r^+}\mathcal{N}^+_f\widetilde{\mathcal{N}}^{-1}_f g^--\frac{1}{\r^-}\mathcal{N}^-_f\widetilde{\mathcal{N}}^{-1}_f g^+\\\nonumber
	=&-\frac{\r^+g^++\r^-g^-}{\r^++\r^-}+\frac{1}{\r^++\r^-}(\mathcal{N}^+_f-\mathcal{N}^-_f)\widetilde{\mathcal{N}}^{-1}_f(g^+-g^-)\\\nonumber
	=&-\frac{2}{\r^++\r^-}
	\big((\r^+\underline{u}^+_1+\r^-\underline{u}^-_1)\partial_1\theta+(\r^+\underline{u}^+_2+\r^-\underline{u}^-_2)\partial_2\theta\big)\\\nonumber
	&-\frac{1}{\r^++\r^-}\sum^2_{s,r=1}\big(\r^+\underline{u}_s^+\underline{u}^+_r- \rho^+\sum^3_{j=1}\underline{F} _{sj}^+\underline{F} _{rj}^+\big)\partial_s\partial_rf\\\nonumber
	&-\frac{1}{\r^++\r^-}\sum^2_{s,r=1}\big(\r^-\underline{u}^-_s\underline{u}^-_r-\rho^-\sum^3_{j=1} \underline{F} _{sj}^-\underline{F} _{rj}^-\big)\partial_s\partial_rf\\\nonumber
	&+\frac{1}{\r^++\r^-}(\mathcal{N}^+_f-\mathcal{N}^-_f)\widetilde{\mathcal{N}}^{-1}_f\mathcal P\big(\sum^2_{s,r=1}\big(\underline{u}_s^+\underline{u}^+_r-\sum^3_{j=1}\underline{F} _{sj}^+\underline{F} _{rj}^+\big)\partial_s\partial_rf\big)\\\nonumber
	&-\frac{1}{\r^++\r^-}(\mathcal{N}^+_f-\mathcal{N}^-_f)\widetilde{\mathcal{N}}^{-1}_f\mathcal P\big(\sum^2_{s,r=1}\big(\underline{u}^-_s\underline{u}^-_r-\sum^3_{j=1}\underline{F} _{sj}^-\underline{F} _{rj}^-\big)\partial_s\partial_rf\big)\\\nonumber
	&+\frac{1}{\r^++\r^-}(\mathcal{N}^+_f-\mathcal{N}^-_f)\widetilde{\mathcal{N}}^{-1}_f\mathcal P
	\big((\underline{u}^+_1-\underline{u}^-_1)\partial_1\theta+(\underline{u}^+_2-\underline{u}^-_2)\partial_2\theta\big)\\\nonumber
	&+\frac{1}{\r^++\r^-}N_f\cdot\underline{\nabla(\r^+p_{\vup, \vup}-\rho^+\sum^3_{j=1}p_{\vF^+_j, \vF^+_j})}\\\nonumber
	&+\frac{1}{\r^++\r^-}N_f\cdot\underline{\nabla(\r^-p_{\vum,\vum}-\rho^-\sum^3_{j=1}p_{\vF^-_j, \vF^-_j})}\\\nonumber
	&-\frac{1}{\r^++\r^-}(\mathcal{N}^+_f-\mathcal{N}^-_f)\widetilde{\mathcal{N}}^{-1}_f\mathcal P N_f\cdot\underline{\nabla\big(p_{\vup, \vup}-\sum^3_{j=1}p_{\vF^+_j, \vF^+_j}-p_{\vum,\vum}+\sum^3_{j=1}p_{\vF^-_j, \vF^-_j}\big)}.\\\label{eq:theta}
\end{align}
Here $\mathcal P:L^2(\bbT^2)\to L^2(\bbT^2)$ is a projection operator defined by
\begin{align*}
	\mathcal Pg=g-\langle g\rangle	
\end{align*}
with $\langle g\rangle \eqdefa\int_{\bbT^2}gdx'$.

We can apply the operator $\mathcal P$ to some of the terms in (\ref{eq:theta}) for the same reasons
as in \cite{SWZ1}, because it does not change the formulation of this system owing to $\mathcal P g^\pm=g^\pm$.

\subsection{Equations for the vorticity and the curl of deformation tensor}
We  derive the equations for
\begin{align*}
	\vom^\pm=\nabla\times \vu^\pm,\qquad \vG_j^\pm=\nabla\times \vF^\pm_j.
\end{align*}
It is direct to obtain from  (\ref{els}) that $(\vom^\pm,\vG_j^\pm)$ satisfies
\begin{equation}\label{eq.vor}
  \left\{
  	\begin{array}{ll}
  		\pa_t\vom^\pm+\vu^\pm\cdot\nabla\vom^\pm-\sum\limits^3_{j=1}\vF^\pm_j\cdot\nabla \vG^\pm_j=\vom^\pm\cdot\nabla \vu^\pm-\sum\limits^3_{j=1}\vG^\pm_j\cdot\nabla \vF^\pm_j&\text{in}\quad \Om_t;\\
  		\pa_t \vG^\pm_j+\vu^\pm\cdot\nabla \vG^\pm_j-\vF^\pm_j\cdot\nabla\vom^\pm=\vG^\pm_j\cdot\nabla \vu^\pm-\vom^\pm\cdot\nabla \vF^\pm_j-2\sum\limits^3_{i=1}\nabla u^\pm_i\times \nabla F^\pm_{ij}&\text{in}\quad \Om_t.
  	\end{array}
  \right.
\end{equation}

\subsection{The evolution of tangential parts of $\vu$ and $\vF_j$ on top and bottom boundaries}
As in \cite{SWZ1}, we derive the evolution of
\begin{align*}
	\b^\pm_i=\int_{\bbT^2}u^\pm_i(t,x',\pm1)dx',\quad \gamma^\pm_{ij}(t)=\int_{\bbT^2}F_{ij}^\pm(t,x',\pm1)dx'\quad \text{for } i=1,2\text{ and } j=1,2,3.
\end{align*}
As $u^\pm_3(t,x',\pm1)\equiv0$, we deduce that for $i=1,2$
\begin{align*}
	\pa_t u^\pm_i+u^\pm_s\pa_s u^\pm_i-\sum^3_{j=1}F^\pm_{sj}\pa_s F^\pm_{ij}-\pa_i p^\pm=0 \qquad \text{on} ~ \Gamma^\pm.
\end{align*}
Consequently, one has
\begin{align*}
	\pa_t\b^\pm_i+\int_{\Gamma^\pm}\big(u^\pm_s\pa_s u^\pm_i-\sum^3_{s,j=1}F^\pm_{sj}\pa_s F^\pm_{ij}\big)dx'=0,
\end{align*}
or equivalently
\begin{align*}
	\b^\pm_i(t)=\b^\pm_i(0)-\int^t_0\int_{\Gamma^\pm}\big(u^\pm_s\pa_s u^\pm_i-\sum^3_{j=1}F^\pm_{sj}\pa_s F^\pm_{ij}\big)dx'd\tau.
\end{align*}
Similarly, we have
\begin{align*}
	\gamma^\pm_{ij}(t)=\gamma^\pm_{ij}(0)-\int^t_0\int_{\Gamma^\pm}\big(u^\pm_s\pa_s F^\pm_{ij}-F^\pm_{sj}\pa_s u_i\big)dx'd\tau.
\end{align*}

\subsection{Solvability conditions of Div-Curl system}

To recover the divergence-free velocity field or deformation tensor field from its curl part, we solve the following div-curl system:
\begin{equation}\label{div-curl-temp}
  \left\{
  	\begin{array}{ll}
  		\curl \vu^\pm=\om^\pm,\quad \div\vu^\pm=g^\pm\quad\text{in}\quad \Omega_f^\pm,\\
		\vu^\pm\cdot\vN_f =\theta^\pm\quad\text{on}\quad\Gamma_{f}, \\
		\vu^\pm\cdot\ve_3 = 0 \text{ on }\Gamma^{\pm}, \quad \int_{\Gamma^\pm} u_i^\pm dx'=\beta^\pm_i (i=1,2).&
  	\end{array}
  \right.
\end{equation}
The solvability of the above system was obtained in \cite{SWZ1} under the following  compatibility conditions:
\begin{itemize}
\item[C1.]  $\div\vom^\pm=0$\, in $\Omega_f^\pm$,
\item[C2.]  $\int_{\Gamma^\pm}\vom_3^\pm dx'=0$,
\item[C3.]  $\int_{\bbT^2}\theta dx'=\mp\int_{\Gamma^\pm} g^\pm dx'$,
\end{itemize}
and the main result are stated in Proposition \ref{prop:div-curl}.

\section{Uniform estimates for the linearized system}

In this section, we will present the uniform energy estimates for the linearized system around given functions $(f,\vu^\pm,\vF^\pm)$.
We assume that there exists $T>0$ for any $t\in [0,T]$:
\begin{align}
	&\|(\vu^\pm, \vF^\pm)(t)\|_{L^{\infty}(\Gamma_f)}\le L_0,\label{ass:regularity}\\
	&\|f(t)\|_{H^{s+\frac{1}{2}}(\bbT^2)}+\|\pa_tf(t)\|_{H^{s-\f12}(\bbT^2)}+\|\vu^\pm(t)\|_{H^{s}(\Omega_f^\pm)}
	+\|\vF^\pm(t)\|_{H^{s}(\Omega_f^\pm)}\le L_1,\\
	&\|(\partial_t\vu^\pm, \partial_t\vh^\pm)(t)\|_{L^{\infty}(\Gamma_f)}\le L_2,\\
	&\|f(t)-f_*\|_{H^{s- \frac{1}{2}}}\le \delta_0,\\
	&-(1-c_0)\le f(t,x')\le (1-c_0),\\
	&\Lambda(\vF^\pm, \vv)(t)\ge c_0,\label{ass:stability}
\end{align}
 together with
\begin{equation}\label{ass:boun}
  \left\{
  	\begin{array}{ll}
  		\div\vu^\pm=\div\vF_j^\pm=0\quad&\text{in}\quad\Omega_f^\pm,\\
		{\vF_j}^\pm\cdot\vN_f=0&\text{on}\quad\Gamma_f,\\
		\partial_tf=\underline{\vu}^\pm\cdot\vN_f,\quad &\text{on}\quad\Gamma_f,\\
		u_3^\pm=F_{3j}^\pm=0\quad&\text{on}\quad \Gamma^\pm.
  	\end{array}
  \right.
\end{equation}
Here $L_0, L_1, L_2,c_0,\delta_0$ are positive constants.

\subsection{The linearized system for the height function of the interface}

For the system (\ref{eq:form:f}) and (\ref{eq:theta}), we introduce the following linearized system:
\begin{equation}\label{sys:linear-H}
  \left\{
  	\begin{array}{l}
  		\pa_t\bar f=\bar\theta;\\
  		\partial_t\bar\theta=-\frac{2}{\r^++\r^-}
		\big((\r^+\underline{u}^+_1+\r^-\underline{u} ^-_1)\partial_1\bar\theta+(\r^+\underline{u} ^+_2+\r^-\underline{u} ^-_2)\partial_2\bar\theta\big)\\
		\quad-\frac{1}{\r^++\r^-}\sum\limits^2_{s,r=1}\big(\r^+\underline{u} _s^+\underline{u} ^+_r- \rho^+\sum\limits^3_{j=1}\underline{F} _{sj}^+\underline{F} _{rj}^+\big)\partial_s\partial_r\bar f\\
		\quad-\frac{1}{\r^++\r^-}\sum\limits^2_{s,r=1}\big(\r^-\underline{u} ^-_s\underline{u} ^-_r-\rho^-\sum\limits^3_{j=1} \underline{F} _{sj}^-\underline{F} _{rj}^-\big)\partial_s\partial_r\bar f+\mathfrak g,\\
  	\end{array}
  \right.
\end{equation}
where
\begin{align}\label{eq:g-def}
	\mathfrak g=&\frac{1}{\r^++\r^-}(\mathcal{N}^+_f-\mathcal{N}^-_f)\widetilde{\mathcal{N}}^{-1}_f\mathcal P\big(\sum^2_{s,r=1}\big(\underline{u} _s^+\underline{u} ^+_r-\underline{u} ^-_s\underline{u} ^-_r+\sum^3_{j=1}
(\underline{F} _{sj}^-\underline{F} _{rj}^--\underline{F} _{sj}^+\underline{F} _{rj}^+)\big)\partial_s\partial_rf\big)\\\nonumber
	&+\frac{1}{\r^++\r^-}(\mathcal{N}^+_f-\mathcal{N}^-_f)\widetilde{\mathcal{N}}^{-1}_f\mathcal P
	\big((\underline{u}^+_1-\underline{u} ^-_1)\partial_1\theta+(\underline{u} ^+_2-\underline{u} ^-_2)\partial_2\theta\big)\\\nonumber
	&+\frac{\rho^+}{\r^++\r^-}N_f\cdot\underline{\nabla(p_{\vup, \vup}-\sum^3_{j=1}p_{\vF^+_j, \vF^+_j})}+\frac{\rho^-}{\r^++\r^-}N_f\cdot\underline{\nabla(p_{\vum,\vum}-\sum^3_{j=1}p_{\vF^-_j, \vF^-_j})}\\\nonumber
	&-\frac{1}{\r^++\r^-}(\mathcal{N}^+_f-\mathcal{N}^-_f)\widetilde{\mathcal{N}}^{-1}_f\mathcal P N_f\cdot\underline{\nabla\big(p_{\vup, \vup}-\sum^3_{j=1}p_{\vF^+_j, \vF^+_j}-p_{\vum,\vum}+\sum^3_{j=1}p_{\vF^-_j, \vF^-_j}\big)}\\\nonumber
	\triangleq&g_1+g_2+g_3+g_4.
\end{align}
Here we need to be careful that $\int_{\bbT^2}\bar\theta dx'$ may not equal to 0.
\begin{remark}
  Let $D_t=\pa_t+w_1\pa_1+w_2\pa_2$. Thus, we have
\begin{align*}
	D_t^2\bar f=\sum\limits^2_{s,r=1}\Big(-v_sv_r+\frac{\rho^+}{\r^++\r^-}\sum\limits^3_{j=1}F^+_{sj}F^+_{rj}+\frac{\rho^-}{\r^++\r^-}\sum\limits^3_{j=1}F^-_{sj}F^-_{rj}\Big)\pa_s\pa_r \bar f+\text{low order terms}.
\end{align*}
The principal symbol of the operator on the right-hand side is
\begin{align}
	(v_i\xi_i)^2- \Big(\frac{\rho^+}{\r^++\r^-}\sum\limits^3_{j=1}(F^+_{ij}\xi_i)^2+\frac{\rho^-}{\r^++\r^-}\sum\limits^3_{j=1}(F^-_{ij}\xi_i)^2\Big).
\end{align}
The negativity of this symbol is ensured by the stability condition (\ref{ass:stability})(see (\ref{condition:s2})).
Therefore, $f$ satisfies a strictly hyperbolic equation, and thus the system should be linearly well-posed.
\end{remark}

Define the energy functional $E_s$ as
\begin{align}\nonumber
	E_s(\partial_t\bar{f},\bar{f})\eqdefa &\big\|(\partial_t+w_i\partial_i)\Ds\bar{f}\big\|_{L^2}^2
	-\big\|v_i\partial_i\Ds\bar{f}\big\|_{L^2}^2\\
	&+\frac{\rho^+}{\r^++\r^-}\sum\limits^3_{j=1}\big\|\underline{F} _{ij}^+\partial_i\Ds\bar{f}\big\|_{L^2}^2
	+\frac{\rho^-}{\r^++\r^-}\sum\limits^3_{j=1}\big\|\underline{F} _{ij}^-\partial_i\Ds\bar{f}\big\|_{L^2}^2,
\end{align}
where $\langle \na\rangle^s f=\mathcal{F}^{-1}((1+|\xi|^2)^{\f s 2}\widehat{f})$ and
\begin{align*}
	w_i=\frac{1}{\r^++\r^-}(\r^+\underline{u} ^+_i+\r^-\underline{u} ^-_i),\qquad v_i=\frac{\sqrt{\r^+\r^-}}{\r^++\r^-}(\underline{u} ^+_i-\underline{u} ^-_i).
\end{align*}
Obviously, there exists $C(L_0)>0$ so that
\begin{align}\label{linear:equi-norm-1}
	E_s(\partial_t\bar{f},\bar{f})\le C(L_0)\left(\|\partial_t\bar{f}\|_{H^{s-\f12}}^2
	+\|\bar{f}\|_{H^{s+\f12}}^2\right).
\end{align}
In addition, we deduce from the stability condition (\ref{ass:stability}) that there exists $C(c_0,L_0)$ so that
\begin{align}\label{linear:equi-norm-2}
	&\|\partial_t \bar{f}\|_{H^{s-\f12}}^2+\|\bar{f}\|_{H^{s+\f12}}^2
	\le C(c_0,L_0)\Big\{E_s(\partial_t\bar{f},\bar{f})+\|\partial_t\bar{f}\|_{L^2}^2+\|\bar{f}\|_{L^2}^2\Big\}.
\end{align}

Firstly, we have the estimate of $\mathfrak{g}$ defined by (\ref{eq:g-def}).
\begin{lemma}\label{lem:non-g}
	It holds that
	\begin{align*}
		\|\mathfrak{g}\|_{H^{s-\f12}}\le C(L_1).
	\end{align*}
\end{lemma}
\begin{proof}
The proof is similar to Lemma 6.2 in \cite{SWZ1}. By using Proposition \ref{prop:DN-Hs} and Proposition \ref{prop:DN-inverse},
we have
\begin{align*}
\|\mathfrak{g}_1\|_{H^{s-\f12}}\le& C(L_1)\left\|(\underline{u}_i^+\underline{u}^+_j-\underline{F}^+_i\underline{F}^+_j
-\underline{u}^-_i\underline{u}^-_j+\underline{F}^-_i\underline{F}^-_j)\partial_i\partial_jf
\right\|_{H^{s-\f32}}\\
\le& C(L_1)\|(\underline{\vu}^\pm,\underline{\vF}^\pm)\|_{H^{s-\f32}}\|f\|_{H^{s+\f12}}\le C(L_1),\\
\|\mathfrak{g}_2\|_{H^{s-\f12}}
\le& C(L_1)\|\underline{\vu}^\pm\|_{H^{s-\f32}}\|\theta\|_{H^{s-\f12}}\le C(L_1),\\
\|(\mathfrak{g}_3, \mathfrak{g}_4)\|_{H^{s-\f12}}
\le& C(L_1)\Big(\|\underline{\nabla(p_{\vum,\vum}-\sum^3_{j=1}p_{\vF^-_j, \vF^-_j})}\|_{H^{s-\f12}}+\|\underline{\nabla(p_{\vup, \vup}-\sum^3_{j=1}p_{\vF^+_j, \vF^+_j})}\|_{H^{s-\f12}}\Big)\\
\le& C(L_1)\Big(\big\|\nabla(p_{\vum,\vum}, \sum^3_{j=1}p_{\vF^-_j, \vF^-_j})\big\|_{H^{s}(\Om_f^-)}
+\big\|\nabla(p_{\vup, \vup}, \sum^3_{j=1}p_{\vF^+_j, \vF^+_j})\big\|_{H^{s}(\Om_f^+)}\Big)\\
\le& C(L_1)\|(\vu^\pm,\vF^\pm)\|_{H^s(\Om_f^\pm)}\le C(L_1).
\end{align*}
The proof is finished.
\end{proof}

Then we have the following estimate.

\begin{proposition}\label{prop:f-L}
	Assume that $\mathfrak{g}\in L^\infty(0,T;H^{s- \frac{1}{2}}(\bbT^2))$. Given the initial data $(\bar \theta_0, \bar f_0)\in H^{s- \frac{1}{2}}\times H^{s+\frac{1}{2}}(\bbT^2)$, there exists a unique solution $(\bar f,\bar \theta)\in C\big([0,T];H^{s+\frac{1}{2}}\times H^{s- \frac{1}{2}}(\bbT^2)\big)$ to the system (\ref{sys:linear-H}) so that
	\begin{align*}
		&\sup_{t\in[0,T]}\left(\|\partial_t\bar{f}(t)\|_{H^{s- \frac{1}{2}}}^2+\|\bar{f}(t)\|_{H^{s+\frac{1}{2}}}^2\right)\\
		&\quad\le C(c_0,L_0)\left(\|\bar\theta_0\|_{H^{s- \frac{1}{2}}}^2+\|\bar f_0\|_{H^{s+\frac{1}{2}}}^2+\int_0^T\|\mathfrak{g}(\tau)\|_{H^{s- \frac{1}{2}}}d\tau\right)e^{C(c_0, L_1,L_2)T}.
	 \end{align*}
\end{proposition}
\begin{proof}
It suffices to prove the uniform estimates. 

From the equation (\ref{sys:linear-H}), we obtain
	\begin{align*}
		\pa^2_t\bar f=&-2(w_1\partial_1\bar\theta+w_2\partial_2\bar\theta)+\sum_{s,t=1,2}(-w_sw_t-v_sv_t)\partial_s\partial_t\bar f\\
		&+\sum^2_{s,r=1}(\frac{\rho^+}{\r^++\r^-}\sum^3_{j=1}\underline{F} _{sj}^+\underline{F} _{rj}^++\frac{\rho^-}{\r^++\r^-}\sum^3_{j=1}\underline{F} _{sj}^-\underline{F} _{rj}^-)\partial_s\partial_r\bar f+\mathfrak g,
	\end{align*}
which yields that
\begin{align*}
		\frac{1}{2}\frac{d}{dt}&\big\|(\partial_t+w_i\partial_i)\Ds\bar{f}\big\|_{L^2(\bbT^2)}^2\\
		=& \Big\langle(\partial_t+w_i\partial_i)\Ds\bar{f}, \Ds\partial_{t}^2\bar{f}+w_i\partial_i(\Ds\partial_t\bar{f})
		+\partial_tw_i\partial_i\Ds\bar{f}\Big\rangle\\
		=& \Big\langle(\partial_t+w_i\partial_i)\Ds\bar{f},\Ds(-2w_i\partial_i\pa_t\bar f+\sum^2_{s,r=1}(-w_sw_r-v_sv_r)\partial_s\partial_r\bar f\\
		&+\sum^2_{s,r=1}(\frac{\rho^+}{\r^++\r^-}\sum^3_{j=1}\underline{F} _{sj}^+\underline{F} _{rj}^++\frac{\rho^-}{\r^++\r^-}\sum^3_{j=1}\underline{F} _{sj}^-\underline{F} _{rj}^-)\partial_s\partial_r\bar f)\Big\rangle\\
		&+\big\langle(\partial_t+w_i\partial_i)\Ds\bar{f},\Ds\mathfrak g+w_i\pa_i(\Ds\pa_t\bar f)+\pa_tw_i\pa_i\Ds\bar f\big\rangle\\
		=&\big\langle(\partial_t+w_i\partial_i)\Ds\bar{f}, -w_i\partial_i\Ds\partial_t\bar{f}\big\rangle\\
		&+\big\langle(\partial_t+w_i\partial_i)\Ds\bar{f},\sum^2_{s,r=1}(-w_sw_r-v_sv_r)\partial_s\partial_r\Ds\bar f\\
		&\qquad+\sum^2_{s,r=1}(\frac{\rho^+}{\r^++\r^-}\sum^3_{j=1}\underline{F} _{sj}^+\underline{F} _{rj}^++\frac{\rho^-}{\r^++\r^-}\sum^3_{j=1}\underline{F} _{sj}^-\underline{F} _{rj}^-)\partial_s\partial_r\Ds\bar f)\big\rangle\\
		&+2\big\langle (\partial_t+w_i\partial_i)\Ds\bar{f},[w_i,\Ds]\pa_i\pa_t\bar f\big\rangle\\
		&+\big\langle (\partial_t+w_i\partial_i)\Ds\bar{f},\\
		&\qquad \big[w_sw_r+v_sv_r -\frac{\rho^+}{\r^++\r^-}\sum^3_{j=1}\underline{F} _{sj}^+\underline{F} _{rj}^+-\frac{\rho^-}{\r^++\r^-}\sum^3_{j=1}\underline{F} _{sj}^-\underline{F} _{rj}^-,\Ds\big]\pa_s\pa_r\bar f \big\rangle \\
		&+\big\langle (\partial_t+w_i\partial_i)\Ds\bar{f},\Ds\mathfrak g+\pa_t w_i\pa_i\Ds\bar f \big\rangle\\
		\triangleq& I_1+\cdots I_5.
	\end{align*}	
From Lemma \ref{lem:commutator}, one has
 	\begin{align*}
		I_3\le& 2\|(\partial_t+w_i\partial_i)\Ds\bar{f}\|_{L^2}\big\|\big[w_i, \Ds\big]\partial_i\partial_t\bar{f}\big\|_{L^2} \\
		\le& CE_s(\partial_t\bar{f},\bar{f})^\f12\|w\|_{H^{s-\f12}}\|\pa_t\bar f\|_{H^{s-\f12}},
	\end{align*}
and
	\begin{align*}
		I_4\le CE_s(\partial_t\bar{f},\bar{f})^{\frac{1}{2}}\Big(\|w\|_{H^{s-\frac{1}{2}}}^2+\|v\|_{H^{s-\frac{1}{2}}}^2+\|\underline{\vF}^\pm\|_{H^{s-\frac{1}{2}}}^2\Big)\|\bar f\|_{H^{s+\frac{1}{2}}}.
	\end{align*}
In addition, it holds that
	\begin{equation}\nonumber
		I_5\le E_s(\partial_t\bar{f},\bar{f})^\f12\big(\|\mathfrak{g}\|_{H^{s-\f12}}+\|\partial_tw\|_{L^\infty}\|\bar f\|_{H^{s+\f12}}\big).
	\end{equation}
It follows from integration by parts that
 	\begin{align*}
 		&\Big\langle\partial_t\Ds\bar{f}, ~-w_i\partial_i\Ds\partial_t\bar{f}\Big\rangle
		\le \|\partial_iw_i\|_{L^\infty}\|\partial_t\bar{f}\|_{H^{s-\f12}}^2,\\
		&\Big\langle w_i\partial_i\Ds\bar{f}, ~-w_i\partial_i\Ds\partial_t\bar{f}\Big\rangle
		+\frac12\frac{d}{dt}\|w_i\partial_i\Ds\bar{f}\|_{L^2}^2\\
		&\quad\qquad=\Big\langle w_i\partial_i \Ds\bar{f},\partial_tw_i\partial_i\Ds\bar{f}\Big\rangle\le
		\|w\|_{L^\infty} \|\partial_tw\|_{L^\infty}\|\bar{f}\|_{H^{s+\f12}}^2,
 	\end{align*}
	which implies
	\begin{equation}
		I_1\le -\frac12\frac{d}{dt}\|w_i\partial_i\Ds\bar{f}\|_{L^2}^2+\big(1+\|w\|_{W^{1,\infty}}+\|\partial_tw\|_{L^\infty}\big)^2
		\Big(\|\bar{f}\|_{H^{s+\f12}}^2+\|\partial_t\bar{f}\|_{H^{s-\f12}}^2\Big).
	\end{equation}
To estimate $I_2$, we can derive
\begin{align*}
&\Big\langle\partial_t\Ds\bar{f}, -w_iw_j\partial_i\partial_j\Ds\bar{f}\Big\rangle
-\frac12\frac{d}{dt}\|w_i\partial_i\Ds\bar{f}\|_{L^2}^2\\
&=-\Big\langle w_i\partial_i\Ds\bar{f}, ~\partial_tw_i
\partial_i\Ds\bar{f}\Big\rangle+\Big\langle\Ds\partial_t\bar{f}, ~\partial_i(w_iw_j)\partial_j\Ds\bar{f}\Big\rangle\\
&\le\|w\|_{L^\infty}\big(\|\partial_tw\|_{L^\infty}+\|\nabla w\|_{L^\infty}\big)
\Big(\|\bar{f}\|_{H^{s+\f12}}^2+\|\partial_t\bar{f}\|_{H^{s-\f12}}^2\Big),
\end{align*}
and similarly
\begin{align*}
&\Big\langle w_k\partial_k\Ds\bar{f}, -w_iw_j\partial_i\partial_j\Ds\bar{f}\Big\rangle
le C \|w\|^2_{L^\infty}\|\nabla w\|_{L^\infty}\|\bar{f}\|_{H^{s+\f12}}^2,
\end{align*}
as well as
	\begin{align*}
		&\Big\langle\partial_t\Ds\bar{f}+w_k\partial_k\Ds\bar{f}, \big(\frac{\rho^+}{\r^++\r^-}\sum^3_{j=1}\underline{F} _{sj}^+\underline{F} _{rj}^++\frac{\rho^-}{\r^++\r^-}\sum^3_{j=1}\underline{F} _{sj}^-\underline{F} _{rj}^-\big)\partial_s\partial_r\Ds\bar{f}\Big\rangle\\
		&\le- \frac{1}{2}\frac{\rho^\pm}{\rho^++\rho^-}\frac{d}{dt}\sum^3_{j=1}\|\underline{F} _{ij}^\pm\pa_i\Ds\bar f\|_{L^2}^2\\\nonumber
		&\quad+C\|\underline{\vF}^\pm\|_{L^\infty}(1+\|\underline{\vF}^\pm\|_{L^\infty})
		\big(\|\partial_t\underline{\vF}^\pm\|_{L^\infty}+\|\nabla \underline{\vF}^\pm\|_{L^\infty}\big)
		\Big(\|\bar{f}\|_{H^{s+\f12}}^2+\|\partial_t\bar{f}\|_{H^{s-\f12}}^2\Big).
	\end{align*}
Therefore, we obtain
	\begin{align*}
		I_2\le& \frac12\frac{d}{dt}\|w_i\partial_i\Ds\bar{f}\|_{L^2}^2+\frac12\frac{d}{dt}\|v_i\partial_i\Ds\bar{f}\|_{L^2}^2\\
		&- \frac{1}{2}\frac{\rho^\pm}{\rho^++\rho^-} \frac{d}{dt}\sum^3_{j=1}\|\underline{F} _{ij}^\pm\pa_i\Ds\bar f\|_{L^2}^2\\
		&+C\big(1+\|(\underline{\vu}^\pm,\underline{\vF}^\pm)\|_{W^{1,\infty}}+\|\pa_t(\underline{\vu}^\pm,\underline{\vF}^\pm)\|_{L^\infty}\big)^3
		\Big(\|\bar{f}\|_{H^{s+\f12}}^2+\|\partial_t\bar{f}\|_{H^{s-\f12}}^2\Big).
	\end{align*}
Combining the estimates of $I_1,\cdots, I_5$ together yields that
	\begin{align*}
		\frac{d}{dt}E_s(&\partial_t\bar{f},\bar{f})\le \|\mathfrak{g}\|_{H^{s-\f12}}^2\\ +&C(L_0)\big(1+\|(\underline{\vu}^\pm,\underline{\vF}^\pm)\|_{H^{s-\f12}}+\|\partial_t(\underline{\vu}^\pm,\underline{\vF}^\pm)\|_{L^\infty}\big)^3\Big(\|\bar{f}\|_{H^{s+\f12}}^2+\|\partial_t\bar{f}\|_{H^{s-\f12}}^2\Big).
	\end{align*}
On the other hand, it is easy to show that
	\begin{align*}
		\frac{d}{dt}\big(\|\partial_t\bar{f}\|_{L^2}^2+\|\bar{f}\|_{L^2}^2\big)\le C(L_0)\Big(\|\bar{f}\|_{H^{s+\f12}}^2+\|\partial_t\bar{f}\|_{H^{s-\f12}}^2\Big)
		+\|\mathfrak{g}\|_{L^2}^2.
	\end{align*}

Let $\mathcal{E}(t)\triangleq\|\bar{f}(t)\|_{H^{s+\f12}}^2+\|\partial_t\bar{f}(t)\|_{H^{s-\f12}}^2$.
It follows from (\ref{linear:equi-norm-2}) that
	\begin{align*}
		\mathcal{E}(t)\le C(c_0,L_0)\Big(&\|\bar\theta_0\|_{H^{s-\f12}}^2+\|\bar f_0\|_{H^{s+\f12}}^2+\int_0^t\|\mathfrak{g}(\tau)\|_{H^{s-\f12}}^2d\tau\\
		&+\int_0^t\big(1+\|(\underline{\vu}^\pm,\underline{\vF}^\pm)(\tau)\|_{H^{s-\f12}}+\|\partial_t(\underline{\vu}^\pm, \underline{\vF}^\pm)(\tau)\|_{L^\infty}\big)^3\mathcal{E}(\tau)d\tau\Big),
	\end{align*}
which together with Lemma \ref{lem:basic} gives
	\begin{align*}
		\mathcal{E}(t)\le C(c_0,L_0)\Big(&\|\bar\theta_0\|_{H^{s-\f12}}^2+\|\bar f_0\|_{H^{s+\f12}}^2+\int_0^t\|\mathfrak{g}(\tau)\|_{H^{s-\f12}}^2d\tau+C(L_1,L_2)\int_0^t\mathcal{E}(\tau)d\tau\Big).
	\end{align*}
By using Gronwall's inequality, we conclude the desired estimate.
\end{proof}

\subsection{The linearized system of $(\vom^\pm,\vG^\pm)$}

For the vorticity system (\ref{eq.vor}), we introduce the following linearized system:
\begin{equation}\label{eq:vorticity-w-L}
  \left\{
  	\begin{array}{l}
  		\pa_t\bar\vom^\pm+\vu^\pm\cdot\nabla \bar\vom^\pm-\sum\limits^3_{j=1}\vF^\pm_j\cdot\nabla \bar{\vG}^\pm_j=\bar{\vom}^\pm\cdot\nabla \vu^\pm-\sum\limits^3_{j=1}\bar{\vG}^\pm_j\cdot\nabla \vF^\pm_j,\\
  		\pa_t\bar \vG^\pm_j+\vu^\pm\cdot\nabla \bar \vG^\pm_j-\vF^\pm_j\cdot\nabla\bar\vom^\pm=\bar \vG^\pm_j\cdot\nabla \vu^\pm-\bar\vom^\pm\cdot\nabla \vF^\pm_j-2\sum\limits^3_{s=1}\nabla u^\pm_s\times\nabla F^\pm_{sj},\\
  \bar\vom^\pm(0,x)=\bar\vom_0^\pm,\qquad \bar\vG^\pm_j(0,x)= \bar\vG_{j,0}^\pm.
  	\end{array}
  \right.
\end{equation}
We first assume the existence of solutions to (\ref{eq:vorticity-w-L}). Then it holds the following estimate.
\begin{proposition}\label{prop:vorticity}
	It holds that
	\begin{align}\label{eq:linearwg}
		\sup_{t\in[0,T]}(\|\bar\vom^\pm(t)\|^2_{H^{s-1}(\Omega^\pm_f)}&+\sum^3_{j=1}\|\bar \vG_j^\pm(t)\|^2_{H^{s-1}(\Omega^\pm_f)})\\
		&\le\Big(1+\|\bar\vom_0^\pm\|^2_{H^{s-1}(\Omega^\pm_f)}+\sum^3_{j=1}\|\bar \vG_{j,0}^\pm\|^2_{H^{s-1}(\Omega^\pm_f)}\Big)e^{C(L_1)T}.\nonumber
	\end{align}
\end{proposition}
\begin{proof}Using $\pa_t f=u^\pm\cdot \vN_f$ and integrating by parts, we obtain
	\begin{align*}
		&\frac{1}{2}\frac{d}{dt}\int_{\Om^\pm_f}|\nabla^{s-1}\bar\vom^\pm(t,x)|^2+\sum^3_{j=1}|\nabla^{s-1}\bar \vG^\pm_j(t,x)|^2dx\\
		=&\int_{\Om^\pm_f}\nabla^{s-1}\bar\vom^\pm\cdot \nabla^{s-1}\pa_t\bar\vom^\pm +\sum^3_{j=1}\nabla^{s-1}\bar \vG_j^\pm\cdot \nabla^{s-1}\pa_t\bar G_j^\pm dx\\
		&\mp\frac{1}{2}\int_{\Gamma_f}(|\nabla^{s-1}\bar\vom^\pm|^2+\sum^3_{j=1}|\nabla^{s-1}\bar G^\pm_j|^2)(\vu^\pm\cdot\vn)d\sigma.
	\end{align*}
From (\ref{eq:vorticity-w-L}) and the fact that $\vF^\pm_j\cdot \vN_f=0$, we can derive
	\begin{align*}
		\frac{1}{2}\frac{d}{dt}&\int_{\Om^\pm_f}|\nabla^{s-1}\bar\vom^\pm(t,x)|^2+\sum^3_{j=1}|\nabla^{s-1}\bar \vG^\pm_j(t,x)|^2dx\\
		\le&\int_{\Om^\pm_f}\nabla^{s-1}\bar\vom^\pm\cdot\nabla^{s-1}[-\vu^\pm\cdot\nabla\bar\vom^\pm]+\sum^3_{j=1}\nabla^{s-1}\bar\vG_j^\pm\cdot\nabla^{s-1}[-\vu^\pm\cdot\nabla \bar\vG_j^\pm]dx\\
		&+\int_{\Om^\pm_f}\nabla^{s-1}\bar\vom^\pm\cdot\nabla^{s-1}[\sum^3_{j=1} \vF_j^\pm\cdot\nabla \bar\vG_j^\pm]+\sum^3_{j=1}\nabla^{s-1}\bar\vom^\pm\cdot\nabla^{s-1}[\vF_j^\pm\cdot\nabla \bar\vG_j^\pm]dx\\
		&\mp\frac{1}{2}\int_{\Gamma_f}(|\nabla^{s-1}\bar\vom^\pm|^2+\sum^3_{j=1}|\nabla^{s-1}\bar \vG^\pm_j|^2)(\vu^\pm\cdot\vn)d\sigma\\
		&+C(L_1)\Big(1+\|\bar\vom^\pm(t)\|^2_{H^{s-1}(\Omega^\pm_f)}+\sum^3_{j=1}\|\bar \vG_j^\pm(t)\|^2_{H^{s-1}(\Omega^\pm_f)}\Big)\\
		\le&\frac{1}{2}\int_{\Om^\pm_f}-\vu^\pm\cdot\nabla(|\nabla^{s-1}\bar\vom^\pm|^2+\sum^3_{j=1}|\nabla^{s-1}\bar \vG_j^\pm|^2)dx\\
		&\mp\frac{1}{2}\int_{\Gamma_f}(|\nabla^{s-1}\bar\vom^\pm|^2+\sum^3_{j=1}|\nabla^{s-1}\bar \vG^\pm_j|^2)(\vu^\pm\cdot\vn)d\sigma\\
		&+\sum^3_{j=1}\int_{\Om^\pm_f}\vF_j^\pm\cdot\nabla(\nabla^{s-1}\bar\vom^\pm\cdot\nabla^{s-1}\bar \vG_j^\pm)dx\\
		&+C(L_1)\Big(1+\|\bar\vom^\pm(t)\|^2_{H^{s-1}(\Omega^\pm_f)}+\sum^3_{j=1}\|\bar \vG_j^\pm(t)\|^2_{H^{s-1}(\Omega^\pm_f)}\Big)\\
		\le&C(L_1)\Big(1+\|\bar\vom^\pm(t)\|^2_{H^{s-1}(\Omega^\pm_f)}+\sum^3_{j=1}\|\bar \vG_j^\pm(t)\|^2_{H^{s-1}(\Omega^\pm_f)}\Big).
	\end{align*}
Then the proposition follows from Gronwall's inequality.
\end{proof}

Now we turn to the existence of solutions of the linearized system (\ref{eq:vorticity-w-L}).\smallskip

First of all, we consider the linear system
\begin{equation}\label{eq.linear}
  \left\{
  	\begin{array}{ll}
  		\pa_t\vom^\pm +\vu^\pm\cdot\nabla\vom^\pm -\sum\limits^3_{j=1}\vF^\pm_j\cdot\nabla \vG^\pm_j=\vg^\pm_0,&(t,x)\in Q_T^\pm,\\
  		\pa_t \vG^\pm_j+\vu^\pm\cdot\nabla \vG^\pm_j-\vF^\pm_j\cdot\nabla \vom^\pm=\vg^\pm_j,&(t,x)\in Q_T^\pm,\\
		\vom^\pm(0,x)=\vom_0^\pm,\qquad \vG^\pm_j(0,x)=\vG^\pm_{j,0}, &x\in\Om^\pm_{f_0}.
  	\end{array}
  \right.
\end{equation}
Here $ Q_T^\pm$ is defined by $f$ in the same way as before.

\begin{lemma}\label{lem:ex}
Assume that $f,\vu^\pm,\vF^\pm$ satisfy (\ref{ass:regularity})-(\ref{ass:stability}). Given the initial data $(\vom_0^\pm,
	\vG_{j,0}^\pm)=(0,0)$, $(\vg^\pm_0,\vg^\pm_j)\in L^1([0,T];H^{s-1}(\Om_f))$, there exists a unique solution $(\vom^\pm,
	\vG^\pm_j)\in C([0,T];H^{s-1}(\Omega^\pm_f)\times H^{s-1}(\Omega^\pm_f))$ to the system (\ref{eq.linear}) satisfying the following estimate
	\begin{align*}
		\sup_{t\in[0,T]}\Big(\|\vom^\pm(t)\|_{H^{s-1}(\Omega^\pm_f)}&+\sum^3_{j=1}\|\vG_j^\pm(t)\|_{H^{s-1}(\Omega^\pm_f)}\Big)\le C(L_1,T)\|(\vg^\pm_0,\vg^\pm_j)\|_{L^1([0,T];H^{s-1}(\Omega^\pm_f))}.
\end{align*}	
\end{lemma}
\begin{proof}
Let $\vW^\pm=(\vom^\pm,\vG^\pm)$. We rewrite the system as
	\begin{align*}
		L(\vW^\pm)=\vg^\pm.
	\end{align*}
We define the flow map $X^\pm(t,\cdot)$ as
	\begin{align*}
		\frac{dX^\pm(\tilde t,\tilde x)}{d\tilde t}=\vu^\pm(\tilde t,X^\pm(\tilde t,\tilde x)),\qquad (\tilde t,\tilde x)\in [0,T]\times\Om^\pm_{f_0} ,
	\end{align*}
	with $t=\tilde t$. Now we write $(t,x)\in Q_T^\pm$ and $(\tilde t,\tilde x)\in [0,T]\times\Om^\pm_{f_0}$.
Then we rewrite $L$ in the new coordinate as
	\begin{align*}
		\widetilde L(\widetilde\vW^\pm)=\pa_{\tilde t}\widetilde \vW^\pm+\widetilde M(\widetilde \vW^\pm)=\widetilde{\vg}^\pm,
	\end{align*}
	where  $\widetilde\vW^\pm(\tilde t, \tilde x)=\vW^\pm\big(\tilde t, X^\pm(\tilde t,\tilde x)\big), \widetilde\vg^\pm(\tilde t, \tilde x)=\vg^{\pm}\big(\tilde t, X^\pm(\tilde t,\tilde x)\big)$, and $\widetilde M$ is given by
	\begin{equation}
		\widetilde M(\widetilde\vW^\pm)=
		\left(
		  	\begin{array}{l}
		  	-\sum\limits^3_{j=1}\big(\widetilde\vF^\pm_j(\tilde t,\tilde x)\cdot \frac{\pa X^{\pm-1}(\tilde t,\cdot)}{\pa\tilde x}\cdot\nabla_{\tilde x}\big)\widetilde\vG^\pm_j(\tilde t,\tilde x)\\
		  	-\big(\widetilde\vF^\pm_1(\tilde t,\tilde x)\cdot \frac{\pa X^{\pm-1}(\tilde t,\cdot)}{\pa\tilde x}\cdot\nabla_{\tilde x}\big)\tilde\vom^\pm(\tilde t,\tilde x)\\
		  	-\big(\widetilde\vF^\pm_2(\tilde t,\tilde x)\cdot \frac{\pa X^{\pm-1}(\tilde t,\cdot)}{\pa\tilde x}\cdot\nabla_{\tilde x}\big)\tilde\vom^\pm(\tilde t,\tilde x)\\
		  	-\big(\widetilde\vF^\pm_3(\tilde t,\tilde x)\cdot \frac{\pa X^{\pm-1}(\tilde t,\cdot)}{\pa\tilde x}\cdot\nabla_{\tilde x}\big)\tilde\vom^\pm(\tilde t,\tilde x)\\
		  	\end{array}
		\right).\nonumber
	\end{equation}

We define
	\begin{align*}
		D\eqdefa\big\{\tilde \vv^\pm=(\tilde\vv^\pm_0,\tilde\vv^\pm_1,\tilde\vv^\pm_2,\tilde\vv^\pm_3)\in C^\infty([0,T]\times\Om^\pm_{f_0})|\tilde\vv^\pm(T,\tilde x)=0\big\}.
	\end{align*}
Then $\vW^\pm$ solves (\ref{eq.linear}) if and only if for every $\tilde\vv^\pm\in D$,
	\begin{align*}
		\int^T_0\int_{\Om_{f_0}^\pm}\widetilde L(\widetilde \vW^\pm)\cdot\tilde\vv^\pm d\tilde xd\tilde t=\int^T_0\int_{\Om_{f_0}^\pm}\tilde \vg^\pm\cdot\tilde \vv^\pm d\tilde xd\tilde t.
	\end{align*}
	Thanks to $\div \vF^\pm_j=0$, $\vF^\pm_j\cdot \vN_f=0$ on $\Gamma_f$, using the flow map $X^\pm(t,\cdot)$, it is easy to show that
$\vW^\pm$ solves (\ref{eq.linear}) if and only if for every $\tilde \vv^\pm\in D$,
	\begin{align}\label{eq:linear-integral}
		\int^T_0\int_{\Om^\pm_{f_0}}\widetilde\vW^\pm\cdot L^*(\tilde \vv^\pm)d\tilde xd\tilde t=\int^T_0\int_{\Om^\pm_{f_0}}\vg^\pm\cdot\tilde\vv^\pm d\tilde xd\tilde t,
	\end{align}
where $L^*$ denotes the dual of $L$, i.e.,
	\begin{equation}
	 	L^*(\tilde \vv)={-}\left(
	  		\begin{array}{l}
		  	\pa_{\tilde t}\tilde\vv^\pm_0-\sum\limits^3_{j=1}\widetilde\vF^\pm_j\cdot \frac{\pa X^{\pm-1}}{\pa\tilde x}\cdot\nabla_{\tilde x}\tilde\vv^\pm_j\\
		  	\pa_{\tilde t}\tilde\vv^\pm_1-\widetilde\vF^\pm_1\cdot \frac{\pa X^{\pm-1}}{\pa\tilde x}\cdot\nabla_{\tilde x}\tilde\vv^\pm_0\\
		  	\pa_{\tilde t}\tilde\vv^\pm_2-\widetilde\vF^\pm_2\cdot \frac{\pa X^{\pm-1}}{\pa\tilde x}\cdot\nabla_{\tilde x}\tilde\vv^\pm_0\\
		  	\pa_{\tilde t}\tilde\vv^\pm_3-\widetilde\vF^\pm_3\cdot \frac{\pa X^{\pm-1}}{\pa\tilde x}\cdot\nabla_{\tilde x}\tilde\vv^\pm_0\\
	  		\end{array}
	 	\right).	\nonumber	
	\end{equation}
	We denote
	\begin{align*}
		L^*(\tilde\vv^\pm)=\widetilde \vV^\pm.
	\end{align*}
It is easy to show that
	\begin{align*}
		\sup_{t\in[0,T]} \|\tilde\vv^\pm\|_{L^2(\Om^\pm_{f_0})}(t) \le C(L_1)\|\widetilde \vV^\pm\|_{L^1([0,T];L^2(\Om^\pm_{f_0}))}.
	\end{align*}	
Hence, the operator $L^*$ is a bijection from $D$ to $L^*(D)$. Let $N_0$ be its inverse.
By Hahn-Banach theorem, we can extend $N_0$(denoted by $N$ its extension) to the space $L^1([0,T];L^2(\Om^\pm_{f_0}))$:	
\begin{align*}
		N:L^1([0,T];L^2(\Om^\pm_{f_0}))\to C([0,T];L^2(\Om^\pm_{f_0})),\qquad \widetilde\vV^\pm\to\tilde \vv^\pm.
	\end{align*}
We denote by $N^*$ the dual of $N$:
	\begin{align*}
		N^*:\mathcal M([0,T];L^2(\Om^\pm_{f_0}))\to L^\infty([0,T];L^2(\Om^\pm_{f_0})),\quad \tilde\vg^\pm\to\widetilde\vW^\pm.
	\end{align*}
Then for $\tilde\vg^\pm\in L^1([0,T];L^2(\Om^\pm_{f_0}))$,  $		\widetilde\vW^\pm=N^*(\tilde\vg^\pm)$ satisfies \eqref{eq:linear-integral} and
	\begin{align*}
		\|\widetilde\vW^\pm\|_{L^\infty([0,T];L^2(\Om^\pm_{f_0}))}\le C(L_1)\|\tilde\vg^\pm\|_{L^1([0,T];L^2(\Om^\pm_{f_0}))}.
	\end{align*}
This proves the existence of the solution.

The regularity of the solution could be proved by using  standard difference quotient method. The uniqueness is obvious.
\end{proof}

Now we consider the system (\ref{eq.linear}) with nonzero initial data
\begin{align*}
	\vom^\pm(0,x)=\vom^\pm_0,\quad \vG^\pm_j(0,x)=\vG^\pm_{j0},
\end{align*}
where $(\vom^\pm_0,\vG^\pm_{j0})\in H^{s-1}(\Omega^\pm_{f_0})\times H^{s-1}(\Omega^\pm_{f_0})$. Let $\widehat\vW^\pm=\vW^\pm-(\vom^\pm_0,\vG^\pm_{j0})$. Then the problem is reduced to the case of zero initial data with $\vg^\pm$ replace by $\vg^\pm-M(\vom^\pm_0,\vG^\pm_{j0})$. From Lemma \ref{lem:ex}, we know that the solution $\widehat\vW^\pm$ exists but with the loss of regularity. To recover the desired regularity, we may first mollify the initial data, and then use the following uniform estimate for smooth solutions:
\begin{align*}
	\sup_{t\in[0,T]}\Big(&||\vom^\pm(t)||^2_{H^{s-1}(\Omega^\pm_f)}+\sum^3_{j=1}||\vG_j^\pm(t)||^2_{H^{s-1}(\Omega^\pm_f)}\Big)\\
	&\le C(L_1,T)\Big(||\vg^\pm||_{L^2([0,T];H^{s-1}(\Om^\pm_f))}+||\vom_0^\pm||^2_{H^{s-1}(\Omega^\pm_f)}+\sum^3_{j=1}||\vG_{j,0}^\pm||^2_{H^{s-1}(\Omega^\pm_f)}\Big).
\end{align*}

Thus, we can conclude the following proposition.

\begin{proposition}\label{prop:ex}
	Assume that $f,\vu^\pm,\vF^\pm$ satisfy (\ref{ass:regularity})-(\ref{ass:stability}). Given the initial data $(\bar\vom_0^\pm,\bar
\vG_{j,0}^\pm)\in H^{s-1}(\Omega^\pm_{f_0})\times H^{s-1}(\Omega^\pm_{f_0})$, there exists a unique solution $(\bar \vom^\pm,
\bar \vG^\pm_j)\in C([0,T];H^{s-1}(\Omega^\pm_f)\times H^{s-1}(\Omega^\pm_f))$ to the system (\ref{eq:vorticity-w-L}) satisfying the estimate (\ref{eq:linearwg}).
\end{proposition}

For the solutions to (\ref{eq:vorticity-w-L}), we also have
\begin{lemma}
	It holds that
	\begin{align*}
		\frac{d}{dt}\int_{\Gamma^\pm}\bar\om^\pm_3dx'=0,\qquad\frac{d}{dt}\int_{\Gamma^\pm}\bar G_{3j}^\pm dx'=0.
	\end{align*}
\end{lemma}
\begin{proof}These are direct consequences of (\ref{eq:vorticity-w-L}) and (\ref{ass:boun}).
From the fact that $\pa_i u^\pm_3=\pa_i F^\pm_{3j}=0$ ($i=1,2$) on $\Gamma^\pm$, we have
	\begin{align*}
		\frac{d}{dt}\int_{\Gamma^+}\bar\om^+_3dx'=&\int_{\Gamma^+}(-u^+_1\pa_1\bar\om^+_3-u^+_2\pa_2\bar\om^+_3+\pa_3\bar\om^+_3)dx'\\
		&+\int_{\Gamma^+}\sum^3_{j=1}(F^+_{1j}\pa_1\bar G^+_{3j}+F^+_{2j}\pa_1\bar G^+_{3j}-\bar G^+_{3j}\pa_3F^+_{3j})dx'\\
		=&\int_{\Gamma^+}(\pa_1 u^+_1+\pa_2 u^+_2+\pa_3 u^+_3)\bar\om^+_3dx'\\
		&-\int_{\Gamma^+}\sum^3_{j=1}(\pa_1 F^+_{1j}+\pa_2 F^+_{2j}+\pa_3 F^+_{3j})\bar G^+_{3j}dx'\\
		=&0.
	\end{align*}
	Similarly, it holds that
	\begin{align*}
		\frac{d}{dt}\int_{\Gamma^+}\bar G^+_{3j}dx'=&-2\int_{\Gamma^+}\sum_i(\pa_1u_i^+\pa_2F^+_{ij}-\pa_2u^+_i\pa_1F^+_{ij})dx'\\
		=&2\int_{\Gamma^+}\sum_i(u_i^+\pa_1\pa_2F^+_{ij}-u^+_i\pa_2\pa_1F^+_{ij})dx'\\
		=&0.
	\end{align*}
The proof for $\bar\om^-_3$, $\bar G^-_{3j}$ is similar.
\end{proof}

\section{Construction and contraction of the iteration map}
We assume that
\begin{align*}
	f_0\in H^{s+\f12}(\bbT^2),\quad  \vu_0^\pm,\,\vF_0^\pm\in H^{s}(\Omega_{f_0}^\pm).
\end{align*}
In addition, we assume that there exists $c_0>0$ such that
\begin{itemize}
	\item[1.] $-(1-2c_0)\le f_0(x')\le (1-2c_0)$;

	\item[2.] $\Lambda(\vF_0^\pm,\vv)\ge 2c_0$.
\end{itemize}

Let $f_*=f_0$, and $\Omega_*^\pm=\Omega_{f_0}^\pm$ be the reference region.
We take the initial data $(f_I,(\partial_tf)_I,\vom_{*I}^\pm,$ $\vG_{*I}^\pm,\beta_{Ii}^\pm,\gamma_{Ii}^\pm)$ for the equivalent system as follows
\begin{align*}
	&f_I=f_0,\quad (\partial_tf)_I=\vu_0^\pm(x',f_0(x'))\cdot(-\partial_1f_0,-\partial_2f_0,1),\\
	&\vom_{*I}^\pm=\curl\vu_0^\pm,\quad \vG_{*I}^\pm=\curl\vF_0^\pm,\\
	&\beta_{Ii}^\pm=\int_{\BT}u_{0i}^\pm(x',\pm1)dx',\quad \gamma_{Iij}^\pm=\int_{\BT}F_{0ij}^\pm(x',\pm1)dx',
\end{align*}
which satisfy
\begin{align}
	\|f_I\|_{H^{s+\f12}}+\|(\vom_{I*}^\pm, \vG_{I*}^\pm)\|_{H^{s-1}(\Omega_*^\pm)}
	+\|(\partial_tf)_I\|_{H^{s-\f12}}+|\beta^\pm_{Ii}|+|\gamma^\pm_{Iij}| \le M_0
\end{align}
for some $M_0>0$.  Then we define the following functional space.
\begin{definition}\label{def:X}
	Given two positive constants $M_1, M_2>0$ with $M_1>2M_0$, we define
	the space $\mathcal{X}=\mathcal{X}(T, M_1, M_2)$ be the collection of
	$(f, \vom_*^\pm, \vG_*^\pm, \beta^\pm_{i},\gamma^\pm_{ij})$, which satisfies
	\begin{align*}
		&\left(f(0),\partial_tf(0), \vom_*^\pm(0), \vG_*^\pm(0),\beta^\pm_{i}(0),\gamma^\pm_{ij}(0)\right)=\big(f_I, (\partial_t f)_I, \vom_{*I}^\pm, \vG_{*I}^\pm, \beta^\pm_{Ii},\gamma^\pm_{Iij}\big),\\
		& \sup_{t\in[0,T]}
		\|f(t,\cdot)-f_*\|_{H^{s-\f12}} \le \delta_0,\\
		&\sup_{t\in[0,T]}\Big(\|f(t)\|_{H^{s+\f12}}+\|\partial_tf(t)\|_{H^{s-\f12}}
		+\|(\vom_*^\pm, \vG_*^\pm)(t)\|_{H^{s-1}(\Omega_*^\pm)}+|\beta^\pm_{i}(t)|+|\gamma^\pm_{ij}(t)|\Big)\le M_1,\\
		&\sup_{t\in[0,T]}\Big(\|\partial_{t}^2f\|_{H^{s-\f32}}
		+\|(\partial_t\vom_*, \partial_t\vG_*)\|_{H^{s-2}(\Omega_*^\pm)}+|\partial_t\beta^\pm_{i}|+|\partial_t\gamma^\pm_{ij}|\Big)\le M_2,
	\end{align*}
	together with the condition $\int_{\bbT^2}\pa_tf(t,x')dx'=0$.
\end{definition}

The main goal of this section is to construct an iteration map $(\bar{f},\bar{\vom}_*, \bar{\vG}_*,\bar{\beta}^\pm_{i},$
$\bar{\gamma}^\pm_{ij})=\mathcal{F}\big(f,\vom_*^\pm, \vG_*^\pm, \beta^\pm_{i},\gamma^\pm_{ij}\big)
\in\mathcal{X}(T, M_1, M_2)$ for given $(f,\vom_*^\pm, \vG_*^\pm,\beta^\pm_{i},\gamma^\pm_{ij})\in \mathcal{X}
(T, M_1, M_2)$  with suitably chosen constants $M_1, M_2$ and $T$. In addition, we will show that the map $\mathcal{F}$
 is contract in $\mathcal{X}(T, M_1, M_2)$ for some suitably chosen $T, M_1, M_2$.

\subsection{Recover the bulk region, velocity and deformation tensor field}
Recall
\begin{align*}
	\Om_f^{+}=\big\{ x \in \Om| x_3 > f (t,x')\big\}, \quad\Om_f^{-}=\big\{ x \in \Om| x_3 < f (t,x')\big\},
\end{align*}
and the harmonic coordinate map $\Phi_f^\pm:\Omega_*^\pm\to\Omega^\pm_f$. Define
\begin{align*}
	\tilde{\vom}^\pm\triangleq P_{f}^{\div} (\vom _*^\pm\circ\Phi_{f }^{-1}),\quad
\tilde{\vG}^\pm\triangleq P_{f}^{\div} (\vG _*^\pm\circ\Phi_{f }^{-1}),
\end{align*}
where $P_{f }^{\div}$ is an project operator which maps a vector field
$\Omega_{f}^\pm$ to its divergence-free part. More precisely,
$P_{f }^{\div}\vom^\pm=\vom^\pm-\nabla\phi^\pm$ with
\begin{equation}\nonumber
  \left\{
  	\begin{array}{ll}
  		\Delta\phi^\pm=\div\vom^\pm\quad&\text{in}\quad \Omega_f^\pm,\\
		\partial_3\phi^\pm=0\quad&\text{on}\quad \Gamma^\pm,\\
		\phi^\pm=0\quad&\text{on}\quad \Gamma_f.
  	\end{array}
  \right.
\end{equation}
Obviously, we have $\div P_{f }^{\div}\vom^\pm=0$ in $\Omega_f^\pm$, and $\ve_3\cdot P_{f}^{\div}\vom^\pm=\om_3^\pm$ on $\Gamma^\pm$. Thus,
$P_{f }^{\div}\vom^\pm$ satisfies conditions (C1) and (C2) on $\Omega_f^\pm$. Following the same arguments, so does $P_{f }^{\div}\vG^\pm$.
Moreover, we have
	\begin{align}
		&\|(\tilde{\vom}^\pm, \tilde{\vG}^\pm)\|_{H^{s-1}(\Omega_f^\pm)}\le C(M_1),\label{eq:wh-est1}\\
		&\|(\pa_t\tilde{\vom}^\pm, \pa_t\tilde{\vG}^\pm)\|_{H^{s-2}(\Omega_f^\pm)}\le C\big(M_1, M_2\big).\label{eq:wh-est2}
	\end{align}
Then we define $\vu^\pm$ and $\vF^\pm$ as the solution of the following system
\begin{equation}
  \left\{
  	\begin{array}{ll}
  		\curl\vu^\pm=\tilde{\vom}^\pm,\quad\div\vu^\pm =0\quad  &\text{in}\quad\Om_f^{\pm},\\
		\vu^\pm\cdot\vN_f=\partial_tf\quad&\text{on}\quad\Gamma_{f}, \\
		 \vu^\pm\cdot\ve_3 = 0,\quad\int_{\Gamma^\pm}u_i dx'=\beta^\pm_i (i=1,2)\quad&\text{on}\quad\Gamma^{\pm},
  	\end{array}
  \right.
\end{equation}
and
\begin{equation}
  \left\{
  	\begin{array}{ll}
  		\curl \vF_j^\pm=\tilde{\vG}_j^\pm,\quad\div\vF_j^\pm=0\quad&\text{in}\quad\Om_f^{\pm},\\
		\vF_j^\pm\cdot\vN_f = 0\quad &\text{on}\quad\Gamma_{f}, \\
		 \vF^\pm\cdot\ve_3 = 0, \quad \int_{\Gamma^\pm} F_{ij} dx'=\gamma^\pm_i (i=1,2)\quad&\text{on}\quad\Gamma^{\pm}.
  	\end{array}
  \right.
\end{equation}
From Proposition \ref{prop:div-curl} and (\ref{eq:wh-est1}), we deduce that
\begin{align}
	\|\vu^\pm\|_{H^{s}(\Omega_f^\pm)}\le &C(M_1)\big(\|\tilde{\vom}^\pm\|_{H^{s-1}(\Omega_f^\pm)}+\|\partial_tf
	\|_{H^{s-\f12}}+|\beta^\pm_1|+|\beta^\pm_2|\big)\le C(M_1),\\
	\|\vF_j^\pm\|_{H^{s}(\Omega_f^\pm)}\le &C(M_1)\big(\|\tilde{\vG}_j^\pm\|_{H^{s-1}(\Omega_f^\pm)}+|\gamma^\pm_{1j}|+|\gamma^\pm_{2j}|\big)\le C(M_1).
\end{align}
Moreover, there holds
\begin{align*}
	\vu^\pm(0)=\vu_0^\pm,\quad \vF^\pm(0)=\vF_0^\pm.
\end{align*}

From the fact that
\begin{align*}
	\partial_t(\vu^\pm\cdot\vN_f)=\pa_t\vu^\pm\cdot\vN_f+\vu^\pm\cdot\partial_t\vN_f
	=(\partial_t\vu^\pm+\partial_3\vu^\pm\partial_tf)\cdot\vN_f+\vu^\pm\cdot\partial_t\vN_f
\end{align*}
on $\Gamma_f$, one can easily deduce that $\partial_t\vu^\pm$ satisfies
\begin{equation}
  \left\{
  	\begin{array}{ll}
  		\curl\partial_t\vu^\pm=\partial_t\tilde{\vom}^\pm,\quad\div\partial_t\vu^\pm=0\quad&\text{in}\quad\Om_f^{\pm},\\
		\partial_t\vu^\pm\cdot\vN_f=\partial_{t}^2f-\partial_tf\partial_3\vu^\pm
		\cdot\vN_f+u_1^\pm\partial_1\partial_tf+u_2^\pm\partial_2\partial_tf
		\quad&\text{on}\quad\Gamma_{f}, \\
		\partial_t\vu^\pm\cdot\ve_3=0, \quad \int_{\Gamma^\pm}\partial_tu_i^\pm dx=\partial_t\beta^\pm_i(i=1,2)\quad&\text{on}\quad\Gamma^{\pm}.
  	\end{array}
  \right.
\end{equation}
By Proposition \ref{prop:div-curl} again and (\ref{eq:wh-est2}), we get
\begin{align*}
	\|\partial_t\vu^\pm\|_{H^{s-1}(\Omega_f^\pm)}\le {C}(M_1, M_2),
\end{align*}
which implies
\begin{align}\nonumber
	\|\vu^\pm(t)\|_{L^\infty(\Gamma_f)}\le& \|\vu^\pm_0\|_{L^\infty(\Gamma_{f_0})}+\int_0^t\|\partial_t\vu^\pm\|_{L^\infty(\Gamma_f)}dt\\
	\le& \frac{M_0}{2}+T{C}(M_1,M_2).\nonumber
\end{align}
Applying similar arguments, we can show that
\begin{align*}
	&\|\partial_t\vF^\pm(t)\|_{H^{s-1}(\Omega_f^\pm)}\le {C}(M_1, M_2),\\
	&\|\vF^\pm(t)\|_{L^\infty(\Gamma_f)}\le \frac{M_0}2+T{C}(M_1,M_2).\label{h-infty}
\end{align*}
Moreover, we have
\begin{align*}
	&\|f(t)-f_0\|_{L^\infty}\le \|f(t)-f_0\|_{H^{s-\f12}}\le T\|\partial_tf\|_{H^{s-\f12}}\le TM_1, \\
	&|\Lambda(\vF^\pm,\vv)-\Lambda(\vF_0^\pm,\vv_0)|\le TC\big(\|\partial_t\vu^\pm\|_{L^\infty(\Gamma_f)},
	\|\partial_t\vF^\pm\|_{L^\infty(\Gamma_f)}\big)\le TC(M_1,M_2).
\end{align*}

Choose $T$ small enough such that
\begin{align*}
	TM_1\le \min\{\delta_0,c_0\},\quad TC(M_1)+T{C}(M_1,M_2)\le \frac {M_0}2,\quad TC(M_1, M_2)\le c_0,
\end{align*}
and $L_0=M_0$, $L_1=M_1$, $L_2={C}(M_1, M_2)$. Then we can obtain that for any $t\in [0,T]$:
\begin{itemize}
	\item $-(1-c_0)\le f(t,x')\le (1-c_0)$;
	\item $\Lambda(\vF^\pm,\vv)(t)\ge c_0$;
	\item $\|(\vu^\pm, \vF^\pm)(t)\|_{L^{\infty}(\Gamma_f)}\le L_0$;
	\item $\|f(t)-f_*\|_{H^{s-\f12}}\le \delta_0$;
	\item $\|f(t)\|_{H^{s+\f12}}+\|\pa_tf(t)\|_{H^{s-\f12}}+\|\vu^\pm(t)\|_{H^{s}(\Omega_f^\pm)}
	+\|\vF^\pm(t)\|_{H^{s}(\Omega_f^\pm)}\le L_1$;
	\item $\|(\partial_t\vu^\pm, \partial_t\vF^\pm)(t)\|_{L^{\infty}(\Gamma_f)}\le L_2$.
\end{itemize}

\subsection{Define the iteration map}

Given $(f,\vu^\pm,\vF^\pm)$ and define initial data as follows:
\begin{align*}
&\left(\bar f_1(0),\bar\theta(0), \bar\vom^\pm(0) , \bar\vG^\pm(0)\right)
=\big(f_0,(\partial_tf)_I, \vom_{*I}^\pm, \vj_{*I}^\pm\big).
\end{align*}
We can solve $\bar f_1$ and $(\bar\vom^\pm, \bar\vG^\pm)$ by the linearized system (\ref{sys:linear-H}) and (\ref{eq:vorticity-w-L}). We define
\begin{align*}
	&\bar\vom_{*}^\pm=\bar\vom^\pm\circ \Phi_{f}^\pm,\quad \bar\vG_{*}^\pm=\bar\vG^\pm\circ\Phi_{f}^\pm,\\
	&\bar\b^\pm_i(t)=\b^\pm_i(0)-\int^t_0\int_{\Gamma^\pm}u_s^\pm\pa_s u^\pm_i-\sum_{j=1}^3F^\pm_{sj}\pa_s F^\pm_{ij}dx'd\tau,\\
	&\bar\gamma^\pm_{ij}(t)=\gamma^\pm_{ij}(0)-\int^t_0\int_{\Gamma^\pm}u^\pm_s\pa_s F^\pm_{ij}-F^\pm_{sj}\pa_s u_i^\pm dx'd\tau.
\end{align*}
Then we have the iteration map $\mathcal{F}$ as follows
\begin{equation}
	\mathcal{F}\big(f,\vom_*^\pm, \vG_*^\pm, \beta^\pm_{i},\gamma^\pm_{ij}\big)
	\eqdefa \big(\bar{f},\bar{\vom}_*^\pm, \bar{\vG}_*^\pm,\bar{\beta}^\pm_{i},
	\bar{\gamma}^\pm_{ij}\big).
\end{equation}
To ensure $\langle \bar f\rangle=\langle f_0\rangle$ and $\int_{\bbT^2}\pa_t \bar f(t,x')dx'=0$ for $t\in [0,T]$, $\bar f$ in the above equation is given by
\begin{equation}
	\bar f(t,x')=\bar f_1(t,x')-\langle \bar f_1\rangle+\langle f_0\rangle.
\end{equation}
\begin{proposition}\label{prop:iteration map}
	There exist $M_1, M_2, T>0$ depending on $c_0, \delta_0, M_0$ so that $\mathcal{F}$ is a map from $\mathcal{X}(T, M_1,M_2)$ to itself.
\end{proposition}
\begin{proof}
	We know that the initial conditions are automatically satisfied according to the Definition \ref{def:X}. From Proposition \ref{prop:f-L} and Proposition \ref{prop:vorticity}, we have
	\begin{align*}\nonumber
		&\sup_{t\in[0,T]}\left( \|\bar{f}(t)\|_{H^{s+\f12}}+\|\partial_t\bar{f}(t)\|_{H^{s-\f12}}+\|\bar\vom_*^\pm(t)\|_{H^{s-1}
		(\Omega_*^\pm)}+\|\bar\vG_*(t)\|_{H^{s-1}(\Omega_*^\pm)}\right)\\
		&\le C(c_0,M_0)e^{C(M_1,M_2)T}.
	\end{align*}
	From the equation (\ref{sys:linear-H}), (\ref{eq:vorticity-w-L}), we deduce that
	\begin{equation}
		\sup_{t\in[0,T]}\Big(\|\partial_{t}^2\bar f\|_{H^{s-\f32}}
		+\|(\partial_t\bar\vom_*, \partial_t\bar\vG_*)\|_{H^{s-2}(\Omega_*^\pm)}\Big)\le C(M_1).		
	\end{equation}
	Obviously, we have
	\begin{align*}
		&|\bar\beta^\pm_i(t)|+|\bar\gamma^\pm_{ij}(t)|\le M_0+TC(M_1),\\
		&|\partial_t\bar\beta^\pm_i(t)|+|\partial_t\bar\gamma^\pm_{ij}(t)|\le C(M_1),\\
		&\|\bar f(t)-f_*\|_{H^{s-\f12}}\le \int_0^t\|\pa_t \bar f(\tau)\|_{H^{s-\f12}}d\tau.		
	\end{align*}
	We firstly take $M_2=C(M_1)$ and then take $M_1$ large enough so that
	\begin{align}
		C(c_0,M_0)<M_1/2.
	\end{align}
	Next, we take $T$ small enough which only depends only on $c_0, \delta_0, M_0$ so that all other conditions in Definition \ref{def:X} are satisfied.
\end{proof}

\subsection{Contraction of the iteration map}
Now we prove the contraction of the iteration map $\mathcal{F}$.
 Let $\big(f^A, \vom_*^{\pm A}, \vG_*^{\pm A},\beta^{\pm A}_{i}, \gamma^{\pm A}_{ij}\big), \big(f^B$, $\vom_*^{\pm B},
 \vG_*^{\pm B}, \beta^{\pm B}_{i}, \gamma^{\pm B}_{ij}\big)\in \mathcal{X}(T, M_1,M_2)$, and
 $\big(\bar f^C,\bar \vom_*^{\pm C}, \bar\vG_*^{\pm C}, \bar\beta^{\pm C}_{i},\bar\gamma^{\pm C}_{ij}\big)=\mathcal{F}
\big(f^C$, $\vom_*^{\pm C}, \vG_*^{\pm C},\beta^{\pm C}_{i}, \gamma^{\pm C}_{ij}\big)$ for $C=A,B$.
In addition, we use $g^D$  to denote the difference $g^A-g^B$.  For instance,
$f^D={f}^A-{f}^B, \vom_*^{\pm D}=\vom_*^{\pm A}-\vom_*^{\pm B}$.

\begin{proposition}\label{prop:contraction}
	There exists $T>0$ depending on $c_0, \delta_0, M_0$ so that
	\begin{align}\nonumber
		\bar E^D\triangleq&~\sup_{t\in[0,T]}\Big(\|\bar{f}^D(t)\|_{H^{s-\f12}}+\|\partial_t\bar{f}^D(t)\|_{H^{s-\f32}}+\|\bar\vom_*^{\pm D}(t)\|_{H^{s-2}(\Omega_*^\pm)}
		\\&\qquad+\|\bar\vG_*^{\pm D}(t)\|_{H^{s-2}(\Omega_*^\pm)}+|\bar\beta^{\pm D}_{i}(t)|+|\bar\gamma^{\pm D}_{ij}(t)|\Big)\nonumber\\\nonumber
		\le&~\frac12\sup_{t\in[0,T]}\Big( \|{f}^D(t)\|_{H^{s-\f12}}+\|\partial_t{f}^D(t)\|_{H^{s-\f32}}
		+\|\vom_*^{\pm D}(t)\|_{H^{s-2}(\Omega_*^\pm)}\\&\qquad\qquad+\|\vG_*^{\pm D}(t)\|_{H^{s-2}(\Omega_*^\pm)}
		+|\beta^{\pm D}_{i}(t)|+|\gamma^{\pm D}_{ij}(t)|\Big)\triangleq E^D.\nonumber
	\end{align}
\end{proposition}
\begin{proof}
Firstly, we have following elliptic estimate
	\begin{align*}
		\|\Phi_{f^A}^\pm-\Phi_{f^B}^\pm\|_{H^{s-1}(\Om_*^\pm)}\le C(M_1)\|f^A-f^B\|_{H^{s-\f12}}\le CE^D.
	\end{align*}
We can not estimate the difference between $\vu^A$ and $\vu^B$ directly, since they are defined on different regions. For this end, we introduce for $C=A,B$,
	\begin{align*}
		\vu^{\pm C}_*=\vu^{\pm C}\circ\Phi_{f^C}^{\pm},\quad
		\vF_{j*}^{\pm C}=\vF_j^{\pm C}\circ\Phi_{f^C}^{\pm}.
	\end{align*}

Now we show that
	\begin{align}\label{eq:uh-d}
		\|\vu^{\pm D}_*\|_{H^{s-1}(\Om^\pm_*)}+\|\vF_{j*}^{\pm D}\|_{H^{s-1}(\Om^\pm_*)}\le CE^D.
	\end{align}
We introduce
	\begin{align*}
		&\curl_C \vv_*^\pm=\big(\curl (\vv_*^\pm\circ(\Phi^{\pm}_{f^C})^{-1})\big) \circ\Phi_{f^C}^\pm,\\
		&\div_C \vv_*^\pm =\big(\div(\vv_*^\pm\circ(\Phi^{\pm}_{f^C})^{-1}\big) \circ\Phi_{f^C}^\pm,
	\end{align*}
for vector field $\vv_*^\pm$ defined on $\Omega_*^\pm$.	Then it holds for $C=A,B$ that
	\begin{equation}
	  \left\{
	  	\begin{array}{ll}
	  		\curl_C \vu^{\pm C}_*=\tilde{\vom}^{\pm C}_*\quad&\text{in}\quad\Om^{\pm}_*,\\
			\div_C \vu_*^{\pm C}=0\quad&\text{in}\quad\Om^{\pm}_*,\\
			\vu^{\pm C}_*\cdot\vN_{f^C}=\partial_tf^C\quad&\text{on}\quad\Gamma_{*},\\
			\vu^{\pm C}\cdot\ve_3 = 0,\quad \int_{\Gamma^\pm} u_i^{\pm C}dx'=\beta_i^{\pm C}\quad&\text{on}\quad\Gamma^{\pm}.
	  	\end{array}
	  \right.
	\end{equation}
Thus, we can deduce
	\begin{equation}
	  \left\{
	  	\begin{array}{ll}
	  		\curl_A\vu^{\pm D}_*=\tilde{\vom}^{\pm D}_*+(\curl_B-\curl_A)\vu^{\pm B}_*\quad&\text{in}\quad \Om^{\pm}_*,\\
			\div_A\vu^{\pm D}_*=(\div_B-\div_A)\vu^{\pm B}_*\quad&\text{in}\quad\Om^{\pm}_*,\\
			\vu^{\pm D}_*\cdot\vN_{f^A}=\partial_tf^D+\vu^{\pm B}_*\cdot(\vN_{f^B}-\vN_{f^A})\quad&\text{on}\quad\Gamma_{*}, \\
			\vu^{\pm D}_*\cdot\ve_3=0,\quad \int_{\Gamma^\pm}u_i^{\pm D}dx'=\beta_i^{\pm D}\quad &\text{on}\quad\Gamma^{\pm}.
	  	\end{array}
	  \right.
	\end{equation}
It is direct to obtain
	\begin{align*}
		\|(\curl_B-\curl_A)\vu^{\pm B}_*\|_{H^{s-2}(\Omega_*^\pm)}\le&~ C\|\Phi_{f^A}^\pm-\Phi_{f^B}^\pm\|_{H^{s-1}(\Omega_*^\pm)}\\
		\le& C\|f^D\|_{H^{s-\f12}}\le CE^D,
	\end{align*}
and similarly,
	\begin{align*}
		&\|(\div_B-\div_A)\vu^{\pm B}_*\|_{H^{s-2}(\Omega_*^\pm)}\le CE^D,\\
		&\|\vu^{\pm B}_*\cdot(\vN_{f^B}-\vN_{f^B})\|_{H^{s-\f32}}\le CE^D.
	\end{align*}
Then applying Proposition \ref{prop:div-curl} yields that
	\begin{align}\nonumber
		\|\vu^{\pm D}_*\|_{H^{s-1}(\Om_*^\pm)}\le C\left(\|\tilde{\vom}^{\pm D}_*\|_{H^{s-2}(\Omega_*^\pm)}
		+\|\partial_tf^D\|_{H^{s-1}}+E^D\right)\le CE^D.
	\end{align}
Similarly, we have
	\begin{align}\nonumber
		\|\vF^{\pm D}_*\|_{H^{s-1}(\Om_*^\pm)}\le& CE^D.
	\end{align}

	Recall that
	\begin{align*}
		\partial_t\bar{f}_1^D=&~\bar\theta^D,\\\nonumber
		\partial_t\bar{\theta}^D=&-\frac{2}{\r^++\r^-}
		\left((\r^+\underline{u} ^{+A}_1+\r^-\underline{u} ^{-A}_1)\partial_1\bar\theta^D+(\r^+\underline{u} ^{+A}_2+\r^-\underline{u} ^{-A}_2)\partial_2\bar\theta^D\right)\\
		&-\frac{1}{\r^++\r^-}\sum^2_{s,r=1}\big(\r^+\underline{u} _s^{+A}\underline{u} ^{+A}_r- \rho^+\sum^3_{j=1}\underline{F} _{sj}^{+A}\underline{F} _{rj}^{+A}\big)\partial_s\partial_r\bar f^D_1\\
		&-\frac{1}{\r^++\r^-}\sum^2_{s,r=1}\big(\r^-\underline{u} ^{-A}_s\underline{u} ^{-A}_r-\rho^-\sum^3_{j=1} \underline{F} _{sj}^{-A}\underline{F} _{rj}^{-A}\big)\partial_s\partial_r\bar f^D_1+\mathfrak R,\label{linear:cauchy:eq-2}
	\end{align*}
	where
	\begin{align}\nonumber
		\mathfrak{R}=&-\frac{2}{\r^++\r^-}
		\left((\r^+\underline{u} ^{+D}_1+\r^-\underline{u} ^{-D}_1)\partial_1\bar\theta^B+(\r^+\underline{u} ^{+D}_2+\r^-\underline{u} ^{-D}_2)\partial_2\bar\theta^B\right)\\ \nonumber
		&-\frac{1}{\r^++\r^-}\sum^2_{s,r=1}\big(\r^+\underline{u} _s^{+A}\underline{u} ^{+A}_r- \rho^+\sum^3_{j=1}\underline{F} _{sj}^{+A}\underline{F} _{rj}^{+A}\big)\partial_s\partial_r\bar f^B_1\\\nonumber
		&+\frac{1}{\r^++\r^-}\sum^2_{s,r=1}\big(\r^+\underline{u} _s^{+B}\underline{u} ^{+B}_r- \rho^+\sum^3_{j=1}\underline{F} _{sj}^{+B}\underline{F} _{rj}^{+B}\big)\partial_s\partial_r\bar f^B_1\\\nonumber
		&-\frac{1}{\r^++\r^-}\sum^2_{s,r=1}\big(\r^-\underline{u} ^{-A}_s\underline{u} ^{-A}_r-\rho^-\sum^3_{j=1} \underline{F} _{sj}^{-A}\underline{F} _{rj}^{-A}\big)\partial_s\partial_r\bar f^B_1\\\nonumber
		&+\frac{1}{\r^++\r^-}\sum^2_{s,r=1}\big(\r^-\underline{u} ^{-B}_s\underline{u} ^{-B}_r-\rho^-\sum^3_{j=1} \underline{F} _{sj}^{-B}\underline{F} _{rj}^{-B}\big)\partial_s\partial_r\bar f^B_1\\\nonumber
		&+\mathfrak{g}^A-\mathfrak{g}^B,\nonumber
	\end{align}
and	for $C=A,B, $
\begin{align} \nonumber
	\mathfrak g^C=&\frac{1}{\r^++\r^-}(\mathcal{N}^+_{f^C}-\mathcal{N}^-_{f^C})\widetilde{\mathcal{N}}^{-1}_{f^C}\Big(\sum^2_{s,r=1}\big(\underline{u} _s^{+C}\underline{u} ^{+C}_r-\sum^3_{j=1}\underline{F} ^{+C}_{sj}\underline{F} _{rj}^{+C}\\\nonumber
	&\quad-\underline{u} ^{-C}_s\underline{u} ^{-C}_r+\sum^3_{j=1}\underline{F} _{sj}^{-C}\underline{F} _{rj}^{-C}\big)\partial_s\partial_r{f^C}\Big)\\\nonumber
	&+\frac{1}{\r^++\r^-}(\mathcal{N}^+_{f^C}-\mathcal{N}^-_{f^C})\widetilde{\mathcal{N}}^{-1}_{f^C}
	\big((\underline{u} ^{+C}_1-\underline{u} ^{-C}_1)\partial_1\theta^C+(\underline{u} ^{+C}_2-\underline{u} ^{-C}_2)\partial_2\theta^C\big)\\\nonumber
	&+\frac{1}{\r^++\r^-}N_{f^C}\cdot\underline{\nabla(\r^+p_{\vu^{+C}, \vu^{+C}}-\rho^+\sum^3_{j=1}p_{\vF^{+C}_j, \vF^{+C}_j})}\\\nonumber
	&+\frac{1}{\r^++\r^-}N_{f^C}\cdot\underline{\nabla(\r^-p_{\vu^{-C},\vu^{-C}}-\rho^-\sum^3_{j=1}p_{\vF^{-C}_j, \vF^{-C}_j})}\\\nonumber
	&-\frac{1}{\r^++\r^-}(\mathcal{N}^+_{f^C}-\mathcal{N}^-_{f^C})\widetilde{\mathcal{N}}^{-1}_{f^C} N_{f^C}\cdot\underline{\nabla\big(p_{\vu^{+C}, \vu^{+C}}-\sum^3_{j=1}p_{\vF^{+C}_j, \vF^{+C}_j}-p_{\vu^{-C},\vu^{-C}}+\sum^3_{j=1}p_{\vF^{-C}_j, \vF^{-C}_j}\big)}.\\\nonumber
\end{align}
	Here $\underline{v}^C(x_1,x_2)$ is the trace of $v$ on $\Gamma_{f^C}$ which interpreted as $v(x_1,x_2, f^C(x_1,x_2))$.

Similar to the proof of Lemma \ref{lem:non-g}, we can show that
\begin{equation*}
	\|\mathfrak{R}\|_{H^{s-\f32}}\le CE^D.
\end{equation*}
We denote
\begin{align}\nonumber
	\bar F_{s}^D(\partial_t\bar{f}_1^D,\bar{f}_1^D)\triangleq&\big\|(\partial_t+w^A_i\partial_i)\langle \na\rangle^{s-\f32}\bar{f}_1^D\big\|_{L^2}^2
	-\big\|v_i^A\partial_i\langle \na\rangle^{s-\f32}\bar{f}_1^D\big\|_{L^2}^2\\
	&+\frac{\rho^+}{\r^++\r^-}\sum\limits^3_{j=1}\big\|F^{+A}_{ij}\partial_i\langle \na\rangle^{s-\f32}\bar{f}_1^D\big\|_{L^2}^2
	+\frac{\rho^-}{\r^++\r^-}\sum\limits^3_{j=1}\big\|F^{-A}_{ij}\partial_i\langle \na\rangle^{s-\f32}\bar{f}_1^D\big\|_{L^2}^2.\nonumber
\end{align}
Following the proof of Proposition \ref{prop:f-L}, one can deduce that
\begin{align*}
	\frac{d}{dt}\Big(\bar F_{s}^D(\partial_t\bar{f}_1^D,\bar{f}_1^D)+\|\bar{f}_1^D\|_{L^2}^2+\|\partial_t\bar{f}_1^D\|_{L^2}^2\Big)
	\le  C\big(E^D+\bar E_1^D)£¬
\end{align*}
where
\begin{align*}
	\bar E_1^D&~=\sup_{t\in[0,T]}\Big(\|\bar{f}_1^D(t)\|_{H^{s-\f12}}+\|\partial_t\bar{f}_1^D(t)\|_{H^{s-\f32}}\Big).
\end{align*}
Recalling the fact that
\begin{align*}
	\|\bar{f}_1^D\|_{H^{s-\f12}}^2+\|\partial_t\bar{f}_1^D\|_{H^{s-\f32}}^2\le C
	\Big(\bar F_{s}^D(\bar{f}_1^D,\partial_t\bar{f}_1^D)+\|\bar{f}_1^D\|_{L^2}^2
	+\|\partial_t\bar{f}_1^D\|_{L^2}^2\Big),
\end{align*}
 we obtain
\begin{align*}
	\sup_{t\in[0,T]}\left(\|\bar{f}_1^D(t)\|_{H^{s-1}}+\|\partial_t\bar{f}_1^D
	(t)\|_{H^{s-\f32}}\right)\le C(e^{CT}-1)E^D,
\end{align*}
which induces
\begin{align}\label{eq:f-d}
	\sup_{t\in[0,T]}\left(\|\bar{f}^D(t)\|_{H^{s-1}}+\|\partial_t\bar{f}^D
	(t)\|_{H^{s-\f32}}\right)\le C(e^{CT}-1)E^D.
\end{align}

Similar to the proof of Proposition \ref{prop:vorticity}, it can be verified that
\begin{align}\label{eq:wj-d}
	\sup_{t\in[0,T]}\left(\|\bar\vom_*^D(t)\|_{H^{s-2}(\Omega_*^\pm)}+\|\bar\vG_*^D\|_{H^{s-2}(\Omega_*^\pm)}\right)\le C(e^{CT}-1)E^D.
\end{align}

From the equation
\begin{align*}
	\bar\b^{\pm C}_i(t)=\bar\b^{\pm C}_i(0)-\int^t_0\int_{\Gamma^\pm}u_s^{\pm C}\pa_s u^{\pm C}_i-\sum^3_{j=1}F^{\pm C}_{sj}\pa_s F^{\pm C}_{ij}dx'd\tau,
\end{align*}
we have
\begin{align}\label{eq:be-d}
	|\bar\beta^{\pm D}_i(t)|\le|\beta^{\pm D}_{iI}|+TCE^D.
\end{align}
It is similar to show that
\begin{align}\label{eq:ga-d}
	|\bar\gamma^{\pm D}_{ij}(t)|\le|\gamma^{\pm D}_{ijI}|+TCE^D.
\end{align}
Thus, thanks to (\ref{eq:uh-d}) and (\ref{eq:f-d})--(\ref{eq:ga-d}), we can conclude that
\begin{align*}
	\bar{E}^D&\le C(e^{CT}-1+T)E^D.
\end{align*}
Taking $T$ small enough depending on $c_0, \delta_0, M_0$, we obtain the proof of the proposition.
\end{proof}
\subsection{The limit system}
It follows from Proposition \ref{prop:iteration map} and Proposition \ref{prop:contraction} that
 there exists a unique fixed point $(f,\vom_*^\pm, \vG_*^\pm,\beta^\pm_i, \gamma^\pm_{ij})$
of the map $\mathcal{F}$  in $\mathcal{X}(T,M_1,M_2)$. In addition,  from the construction of $\mathcal{F}$, we have that
$(f,\vom^\pm, \vG^\pm,\beta^\pm_i, \gamma^\pm_{ij})
=(f,\vom_*^\pm\circ\Phi_f^{-1}, \vG_*^\pm\circ\Phi_f^{-1},\beta^\pm_i, \gamma^\pm_{ij})$ satisfies
\begin{align}
	\partial_tf=&\mathcal P\theta\label{eq:limit-theta},\\\nonumber
	\partial_t\theta=&-\frac{2}{\r^++\r^-}
	\big((\r^+\underline{u}^+_1+\r^-\underline{u} ^-_1)\partial_1\theta+(\r^+\underline{u} ^+_2+\r^-\underline{u} ^-_2)\partial_2\theta\big)\\\nonumber
	&-\frac{1}{\r^++\r^-}\sum^2_{s,r=1}\big(\r^+\underline{u} _s^+\underline{u} ^+_r- \rho^+\sum^3_{j=1}\underline{F} _{sj}^+\underline{F} _{rj}^+\big)\partial_s\partial_rf\\\nonumber
	&-\frac{1}{\r^++\r^-}\sum^2_{s,r=1}\big(\r^-\underline{u} ^-_s\underline{u} ^-_r-\rho^-\sum^3_{j=1} \underline{F} _{sj}^-\underline{F} _{rj}^-\big)\partial_s\partial_rf\\\nonumber
	&+\frac{1}{\r^++\r^-}(\mathcal{N}^+_f-\mathcal{N}^-_f)\widetilde{\mathcal{N}}^{-1}_f\mathcal P\big(\sum^2_{s,r=1}\big(\underline{u} _s^+\underline{u} ^+_r-\sum^3_{j=1}\underline{F} _{sj}^+\underline{F} _{rj}^+\big)\partial_s\partial_rf\big)\\\nonumber
	&-\frac{1}{\r^++\r^-}(\mathcal{N}^+_f-\mathcal{N}^-_f)\widetilde{\mathcal{N}}^{-1}_f\mathcal P\big(\sum^2_{s,r=1}\big(\underline{u} ^-_s\underline{u} ^-_r-\sum^3_{j=1}\underline{F} _{sj}^-\underline{F} _{rj}^-\big)\partial_s\partial_rf\big)\\\nonumber
	&+\frac{1}{\r^++\r^-}(\mathcal{N}^+_f-\mathcal{N}^-_f)\widetilde{\mathcal{N}}^{-1}_f\mathcal P
	\big((\underline{u}^+_1-\underline{u} ^-_1)\partial_1\theta+(\underline{u} ^+_2-\underline{u} ^-_2)\partial_2\theta\big)\\\nonumber
	&+\frac{1}{\r^++\r^-}N_f\cdot\underline{\nabla(\r^+p_{\vup, \vup}-\rho^+\sum^3_{j=1}p_{\vF^+_j, \vF^+_j})}\\\nonumber
	&+\frac{1}{\r^++\r^-}N_f\cdot\underline{\nabla(\r^-p_{\vum,\vum}-\rho^-\sum^3_{j=1}p_{\vF^-_j, \vF^-_j})}\\ \label{limit-f}
	&-\frac{1}{\r^++\r^-}(\mathcal{N}^+_f-\mathcal{N}^-_f)\widetilde{\mathcal{N}}^{-1}_f\mathcal P N_f\cdot\underline{\nabla\big(p_{\vup, \vup}-\sum^3_{j=1}p_{\vF^+_j, \vF^+_j}-p_{\vum,\vum}+\sum^3_{j=1}p_{\vF^-_j, \vF^-_j}\big)},
\end{align}
where $(\vu^\pm,\vF^\pm)$ sovles the div-curl system
\begin{equation}
  \left\{
  	\begin{array}{ll}
  		\curl \vu^\pm=P_f^{div}\vom^\pm,\quad \div\vu^\pm=0\quad &\text{in} \quad \Om_f^\pm,\\
		\vu^\pm\cdot\vN_f=\partial_tf\quad&\text{on}\quad\Gamma_f,\\
		u_3^\pm=0\quad&\text{on}\quad\Gamma^\pm,\\
		\int_{\Gamma^\pm}u_i^\pm dx'=\beta_i^\pm,&\\
		\partial_t\beta_i^\pm~=-\int_{\Gamma^\pm}(u_j^\pm\partial_ju_i^\pm-\sum\limits _{j=1}^3F^\pm_{sj}\pa_s F^\pm_{ij})dx',&
  	\end{array}
  \right.
\end{equation}
and
\begin{equation}
  \left\{
  	\begin{array}{ll}
  		\curl\vF_j^\pm=P_f^{div}\vG_j^\pm,\quad \div \vF_j^\pm=0 &\text{in} \quad \Om_f^\pm,\\
		\vF_j^\pm\cdot\vN_f=0 &\text{on}\quad\Gamma_f,\\
		F_{3j}^\pm=0&\text{on}\quad\Gamma^\pm,\\
		\int_{\Gamma^\pm}F_{ij}^\pm dx'=\gamma_{ij}^\pm,\\
		\partial_t\gamma_{ij}^\pm=-\int_{\Gamma^\pm}(u^\pm_s\pa_s F^\pm_{ij}-F^\pm_{sj}\pa_s u^\pm_i)dx'.
  	\end{array}
  \right.
\end{equation}
and
\begin{align}
&	\pa_t\vom^\pm+\vu^\pm\cdot\nabla \vom^\pm-\sum^3_{i=1}\vF^\pm_i\cdot\nabla {\vG^\pm_i}=\vom^\pm\cdot\nabla \vu^\pm-\sum^3_{i=1}{\vG^\pm_i}\cdot\nabla \vF^\pm_i,\\
\label{limvor}
&	\pa_t \vG^\pm_j+\vu^\pm\cdot\nabla  \vG^\pm_j-\vF^\pm_j\cdot\nabla\vom^\pm= \vG^\pm_j\cdot\nabla \vu^\pm-\vom^\pm\cdot\nabla \vF^\pm_j-2\sum^3_{s=1}\nabla u^\pm_s\times\nabla F^\pm_{sj}.
\end{align}
Here we recall that $p_{\vu_1^\pm,\vu_2^\pm}$ in (\ref{limit-f}) is defined by
\ben\nonumber
\left\{
\begin{array}{l}
\Delta p_{\vu_1^\pm,\vu_2^\pm}= -\mathrm{tr}(\nabla\vu_1^\pm\nabla\vu_2^\pm)
\quad \text{in}\quad\Omega^\pm_f,\\
p_{\vu_1^\pm, \vu_2^\pm}=0\quad\text{on}\quad\Gamma_f,
\quad\\
\ve_3\cdot\nabla p_{\vu_1^\pm,
\vu_2^\pm}=0\quad\text{on}\quad\Gamma^\pm.
\end{array}\right.
\een

To finish the proof of Theorem \ref{thm:1}, we need to show that  the limit system (\ref{eq:limit-theta})-(\ref{limvor}) is equivalent to the original system (\ref{els})-(\ref{elsi}).
We introduce the pressure $p^\pm$ of the fluid by
\begin{align*}
p^\pm=\mathcal{H}_f^\pm\underline{p}^\pm+\rho^\pm p_{\vupm, \vupm}-\rho^\pm\sum^3_{j=1}p_{\vF^\pm_j, \vF^\pm_j},
\end{align*}
where
\begin{align*}
\underline{p}^+=\underline{p}^-=\widetilde{\mathcal{N}}^{-1}_f\mathcal{P}(g^+-g^-)
\end{align*}
with
\begin{align*}
g^\pm=&~2(\underline{u}^\pm_1\partial_1\theta+\underline{u}^\pm_2\partial_2\theta)+\vN\cdot\underline{\nabla(p_{\vupm, \vupm}-p_{\vhpm, \vhpm})}
+\sum_{i,j=1}^2\big(\underline{u}_i^\pm \underline{u}^\pm_j-\sum^3_{l=1}\underline{F} ^\pm_{il}\underline{F} ^\pm_{jl}\big)\partial_i\partial_jf.
\end{align*}
The key idea to prove the consistence is to show that
\ben\label{eq:limit-u}
\left\{
\begin{array}{l}
\div\vw^\pm=0,\quad \curl\vw^\pm=0\quad\text{in}\quad\Omega^\pm_f,\\
\vw^\pm\cdot\vN_f=0\quad\text{on}\quad\Gamma_f,\\
w_3^\pm=0\quad\text{on}\quad \Gamma^\pm,\quad\int_{\Gamma^\pm}w_i^\pm dx'=0(i=1,2).
\end{array}\right.
\een
for
\begin{align*}
\vw^\pm=\partial_t\vu^\pm+\vu^\pm\cdot\nabla\vu^\pm-\sum^3_{j=1}\vF^\pm_j\cdot\nabla\vF^\pm_j+\nabla p^\pm,
\end{align*}
or
\begin{align*}
\vw^\pm=\pa_t  \vF_j +  \vu\cdot\nabla \vF_j - \vF_j\cdot\nabla\vu, \quad j=1,2,3.
\end{align*}
The proof of (\ref{eq:limit-u}) can be accomplished by following \cite[Section 9]{SWZ1} line by line, so we omit the details here.

\section{Proof of Theorem \ref{thm:2}}
In this section, we consider the system (\ref{elsf})-(\ref{elsfi}). Since the  proof of Theorem \ref{thm:2} is
quite analogous to the proof of Theorem \ref{thm:1}, we only present  main steps which are different from the problem  (\ref{els})-(\ref{elsi}).

We first note that the stability condition  ${\rm{rank}}(\vF)=2$ is equivalent to  (\ref{condition:s1}) with $\rho^+=0$ and $\vF^-=\vF$, which further implies that there exists $c_0>0$ such that
\begin{align}\label{condition:2}
\Lambda(\vF)\eqdefa&\inf_{x\in\Gamma_t}\inf_{\ph_1^2+\ph_2^2=1}\sum^3_{j=1}(F_{1j}\ph_1+F_{2j}\ph_2)^2\ge c_0.
\end{align}
Following the derivation of (\ref{eq:theta-d}), one can deduce that
\begin{align}
  \partial_tf=&\theta,\\
\partial_t\theta
=&-2(\underline{u}_1\partial_1\theta+\underline{u}_2\partial_2\theta)-\frac{1}{\r}\vN\cdot\underline{\nabla p}-\sum^2_{s,r=1}
	\underline{u}_s\underline{u}_r\partial_s\partial_rf+\sum^3_{j=1}\sum^2_{s,r=1}\underline{F}_{sj}\underline{F}_{rj}\partial_s\partial_rf.
\end{align}
with $p=\sum\limits^3_{j=1}p_{\vF_j,\vF_j}-p_{\vu,\vu}.$

By the stability condition (\ref{condition:2}), we obtain
\begin{align}\nonumber
	E_s(\partial_tf,f)\eqdefa &~\big\|(\partial_t+u_i\partial_i)\Ds f\big\|_{L^2}^2+\sum\limits^3_{j=1}\big\|\underline{F}_{ij}\partial_i\Ds f\big\|_{L^2}^2\\
\ge &~\big\|(\partial_t+u_i\partial_i)\Ds f\big\|_{L^2}^2+c_0\sum_{i=1}^2\big\|\partial_i\Ds f\big\|_{L^2}^2. \nonumber
\end{align}
Consequently, it holds that
\begin{align*}
	&\|\partial_t f\|_{H^{s-\f12}}^2+\|f\|_{H^{s+\f12}}^2
	\le C(c_0,L_0)\Big\{E_s(\partial_t f,f)+\|\partial_t f\|_{L^2}^2+\|f\|_{L^2}^2\Big\},
\end{align*}
which is actually (\ref{linear:equi-norm-2}). Then, the remain parts of the proof can follow the proof of Theorem \ref{thm:1} step by step.


\begin{appendix}
  \section{}

\subsection{Div-Curl system}
From Section 5 of \cite{SWZ1}, we know that for each div-curl system
\begin{equation}\label{eq:div-curl}
  \left\{
  	\begin{array}{ll}
  		\curl \vu =\vom,\quad\div \vu=g& \text{ in }\quad\Om_f^+,\\
\vu\cdot\vN_f =\vartheta& \text{ on}\quad \Gamma_f,   \\
\vu\cdot\ve_3 = 0,\quad \int_{\bbT^2} u_i dx'=\alpha_i (i=1,2)& \text{ on}\quad\Gamma^{+}.
  	\end{array}
  \right.
\end{equation}
with $f\in H^{s+\frac{1}{2}}(\bbT^2)$ for $s\ge 2$ and satisfying
\begin{align*}
	-(1-c_0)\le f\le(1-c_0),
\end{align*}
have a unique solution.

\begin{proposition}\label{prop:div-curl}
	Let $\sigma \in [2,s]$ be an integer. Given $\vom, g\in H^{\sigma-1}(\Omega_f^+)$, $\vartheta\in H^{\sigma-\frac12}(\Gamma_f)$ with the compatiblity condition:
	\begin{align*}
	  \int_{\Om_f^+} g dx=\int_{\Gamma_f} \vartheta ds,
	\end{align*}
	and $\vom$ satisfies
	\begin{align*}
		&\div\vom=0\quad \text{in}\quad \Omega_f^+,\quad\int_{\Gamma^+}\om_3dx'=0,
	\end{align*}
	Then there exists a unique $\vu\in H^{\sigma}(\Omp)$ of the div-curl system (\ref{eq:div-curl}) so that
	\begin{align*}
		\|\vu\|_{H^{\sigma}(\Omega_f^+)}\le C\big(c_0,\|f\|_{H^{s+\f12}}\big)\Big(\|\vom\|_{H^{\sigma-1}(\Omega_f^+)}+\|g\|_{H^{\sigma-1}(\Omega_f^+)}
		+\|\vartheta\|_{H^{\sigma-\frac12}(\Gamma_f)}+|\alpha_1|+|\alpha_2|\Big).
	\end{align*}
\end{proposition}

\subsection{Commutator estimate}

\begin{lemma}\label{lem:commutator}
	If $s>1+\frac{d}{2}$, then we have
	\begin{equation}
		\big\|[a, \langle\na\rangle^s]u\big\|_{L^2}\le C\|a\|_{H^{s}}\|u\|_{H^{s-1}}.
	\end{equation}
\end{lemma}

\subsection{Sobolev estimates of DN operator}
\begin{proposition}\label{prop:DN-Hs}
	If $f\in H^{s+\frac{1}{2}}(\bbT^2)$ for $s>\frac{5}{2}$, then it holds that for any $\sigma\in \big[-\frac{1}{2},s- \frac{1}{2}	\big]$,
	\begin{equation}
		\|\mathcal{N}^\pm_f\psi\|_{H^{\sigma}}\le  K_{s+\frac{1}{2},f}\|\psi\|_{H^{\sigma+1}}.		
	\end{equation}
	Moreover, it holds that for any $\sigma\in \big[\frac{1}{2},s- \frac{1}{2}	\big]$,
	\begin{equation}
		\|\big(\mathcal{N}^+_f-\mathcal{N}^-_f\big)\psi\|_{H^\sigma}\le K_{s+\frac{1}{2},f}\|\psi\|_{H^\sigma},
	\end{equation}
	where $K_{s+\frac{1}{2},f}$ is a constant depending on $c_0$ and $||f||_{H^s}$.
\end{proposition}
\begin{proposition}\label{prop:DN-inverse}
	If $f\in H^{s+\frac{1}{2}}(\bbT^2)$ for $s>\frac{5}{2}$, then it holds that for any $\sigma\in \big[-\frac{1}{2},s- \frac{1}{2}\big]$,
	\begin{equation}
		\|\mathcal{G}^\pm_f\psi\|_{H^{\sigma+1}}\le  K_{s+\frac{1}{2},f}\|\psi\|_{H^{\sigma}},
	\end{equation}
	where $\mathcal{G}^\pm_f\triangleq\big(\mathcal{N}^\pm_f\big)^{-1}$.
\end{proposition}


\end{appendix}


\begin{thebibliography}{99}

 \bibitem{ABZ} T. Alazard, N. Burq and C. Zuily, {\it On the Cauchy problem for gravity water waves},
 Invent. Math.,  198(2014), 71-163.

\bibitem{AM}  D.  M. Ambrose and N. Masmoudi, {\it Well-posedness of 3D vortex sheets with surface tension}, Commun. Math. Sci.,  5 (2007) 391-430.



\bibitem{Ax} W. I. Axford, {\it Note on a problem of magnetohydrodynamic stability}, Canad. J. Phys., 40(1962), 654-655.


\bibitem{Chen} G.-Q. Chen and Y.-G Wang, {\it Existence and stability of compressible current-vortex sheets in three-dimensional magnetohydrodynamics},
 Arch. Ration. Mech. Anal., 187(2008), 369-408.

\bibitem{CHW} R. M. Chen, J. Hu and D. Wang, {\it Linear stability of compressible vortex sheets in
two-dimensional elastodynamics}, Advances in Mathematics, 311 (2017), 18¨C60.

\bibitem{CCS} C. A. Cheng, D. Coutand and S. Shkoller, {\it On the motion of vortex sheets with surface tension in three-dimensional Euler equations with vorticity}, Comm. Pure Appl. Math., 61 (2008), 1715-1752.



\bibitem{CMST} J.-F. Coulombel, A. Morando, P.  Secchi and P. Trebeschi, {\it A priori estimates for 3D incompressible current-vortex sheets},
              Comm. Math. Phys.,  311(2012), 247-275.


\bibitem{GW}  X. Gu and Y. Wang, {\it On the construction of solutions to the free-surface incompressible ideal magnetohydrodynamic equations}, arXiv:1609.07013.

\bibitem{Hao} C. Hao, {\it On the motion of free interface in ideal incompressible MHD}, Arch. Rational Mech. Anal. (2017)

\bibitem{HL} C. Hao and T. Luo, {\it A priori estimates for free boundary problem of incompressible inviscid magnetohydrodynamic flows},
             Arch. Ration. Mech. Anal.,  212(2014), 805-847.

\bibitem{HW} C. Hao and D. Wang, {\it A priori estimates for the free boundary problem of incompressible neo-Hookean elastodynamics}, J.Differential Equations, 261(2016) 712-737.


\bibitem{Maj} A. Majda and A.  Bertozzi, {\it Vorticity and incompressible flow}, Cambridge Texts in Applied Mathematics, 27,
 Cambridge University Press, Cambridge, 2002.


\bibitem{MTT1} A. Morando, Y. Trakhinin and P. Trebeschi,  {\it Stability of incompressible current-vortex sheets},
J. Math. Anal. Appl., 347(2008), 502-520.

\bibitem{MTT2}  A. Morando, Y. Trakhinin and P. Trebeschi,  {\it Well-posedness of the linearized plasma-vacuum interface problem in ideal incompressible MHD}, Quart. Appl. Math., 72(2014), 549-587.


\bibitem{Sec} P. Secchi and Y. Trakhinin, {\it Well-posedness of the plasma-vacuum interface problem},
             Nonlinearity, 27(2014), 105-169.

\bibitem{SZ1} J. ~Shatah and C. Zeng,  {\it Geometry and  a priori estimates for free boundary problems of the Euler's equation}, Comm. Pure Appl. Math., 61(2008), 698-744.

\bibitem{SZ2} J. ~Shatah and C. Zeng,  {\it A priori estimates for fluid interface problems}, Comm. Pure Appl. Math., 61(2008), 848-876.


\bibitem{SWZ1} Y. Sun, W. Wang and Z. Zhang, {\it Nonlinear stability of the current-vortex sheet to the incompressible MHD equations}, Comm. Pure Appl. Math., 71(2018), 356-403.

\bibitem{SWZ2} Y. Sun, W. Wang and Z. Zhang, {\it Well-posedness of the plasma-vacuum interface problem for ideal incompressible MHD}, 	arXiv:1705.00418.

\bibitem{Sy} S. I. Syrovatskij, {\it The stability of tangential discontinuities in a magnetohydrodynamic medium},
 Z. Eksperim. Teoret. Fiz., 24(1953), 622-629.

\bibitem{ST} P. Secchi and Y. Trakhinin, {\it Well-posedness of the plasma-vacuum interface problem},  Nonlinearity, 27(2014), 105-169.


\bibitem{Tra-in} Y. Trakhinin, {\it On the existence of incompressible current-vortex sheets: study of a linearized free boundary value problem},
           Math. Methods Appl. Sci.,  28(2005), 917-945.

\bibitem{Tra1} Y. Trakhinin, {\it Existence of compressible current-vortex sheets: variable coefficients linear analysis}, Arch. Ration. Mech. Anal., 177(2005), 331-366.


\bibitem{Tra2} Y.  Trakhinin,  {\it The existence of current-vortex sheets in ideal compressible magnetohydrodynamics}, Arch. Ration. Mech. Anal., 191(2009), 245-310.

\bibitem{Tra-JDE} Y.  Trakhinin, {\it On the well-posedness of a linearized plasma-vacuum interface problem in ideal compressible MHD}, { J. Differential Equations}, {249}(2010), 2577-2599.

\bibitem{Tra3} Y. Trakhinin,{\it Well-posedness of the free boundary problem in compressible elastodynamics}, arXiv:1705.11120v3

\bibitem{WY}  Y.-G. Wang and F. Yu,  {\it Stabilization effect of magnetic fields on two-dimensional compressible current-vortex sheets}, Arch. Ration. Mech. Anal.,  208(2013), 341-389.

\bibitem{Wu1} S. Wu, {\it Well-posedness in Sobolev spaces of the full water wave problem in $2$-D},  Invent. Math., 130(1997), 39--72.

\bibitem{Wu2} S. Wu, {\it Well-posedness in Sobolev spaces of the full water wave problem in 3-D}, J. ~Amer. ~Math. ~Soc., 12(1999), 445-495.

\bibitem{ZZ} P. Zhang and Z. Zhang, {\it On the free boundary problem of  three-dimensional incompressible Euler equations},
Comm. Pure Appl. Math., 61(2008), 877--940.

\end{thebibliography}
\end{document}